\documentclass[reqno]{amsart}

\usepackage[T1]{fontenc}
\usepackage{lmodern}
\usepackage{amsmath,amssymb,amsthm,mathtools}
\usepackage{microtype}
\usepackage{cite}
\usepackage{xcolor}
\usepackage{hyperref}
\usepackage[noabbrev]{cleveref}

\numberwithin{equation}{section}
\allowdisplaybreaks

\theoremstyle{plain}
\newtheorem{theorem}{Theorem}[section]
\newtheorem{lemma}[theorem]{Lemma}
\newtheorem{proposition}[theorem]{Proposition}
\newtheorem{corollary}[theorem]{Corollary}
\theoremstyle{definition}
\newtheorem{definition}[theorem]{Definition}
\theoremstyle{remark}
\newtheorem{remark}[theorem]{Remark}

\crefname{theorem}{Theorem}{Theorems}
\Crefname{theorem}{Theorem}{Theorems}
\crefname{lemma}{Lemma}{Lemmas}
\Crefname{lemma}{Lemma}{Lemmas}
\crefname{proposition}{Proposition}{Propositions}
\Crefname{proposition}{Proposition}{Propositions}
\crefname{corollary}{Corollary}{Corollaries}
\Crefname{corollary}{Corollary}{Corollaries}
\crefname{definition}{Definition}{Definitions}
\Crefname{definition}{Definition}{Definitions}
\crefname{remark}{Remark}{Remarks}
\Crefname{remark}{Remark}{Remarks}
\crefname{section}{Section}{Sections}
\Crefname{section}{Section}{Sections}

\crefformat{section}{Section~#2#1#3}
\Crefformat{section}{Section~#2#1#3}
\crefformat{subsection}{Section~#2#1#3}
\Crefformat{subsection}{Section~#2#1#3}
\crefformat{theorem}{Theorem~#2#1#3}
\Crefformat{theorem}{Theorem~#2#1#3}
\crefformat{lemma}{Lemma~#2#1#3}
\Crefformat{lemma}{Lemma~#2#1#3}
\crefformat{proposition}{Proposition~#2#1#3}
\Crefformat{proposition}{Proposition~#2#1#3}
\crefformat{corollary}{Corollary~#2#1#3}
\Crefformat{corollary}{Corollary~#2#1#3}
\crefformat{definition}{Definition~#2#1#3}
\Crefformat{definition}{Definition~#2#1#3}
\crefformat{remark}{Remark~#2#1#3}
\Crefformat{remark}{Remark~#2#1#3}

\newcommand{\R}{\mathbb R}
\newcommand{\M}{\mathcal M}
\newcommand{\I}{\mathbf I}
\newcommand{\differential}{\mathop{}\!\mathrm d}
\newcommand{\Tail}{\operatorname{Tail}}
\DeclareMathOperator{\divergence}{div}

\newcommand{\Lb}{L_b}
\newcommand{\rev}[1]{#1}

\makeatletter
\@namedef{subjclassname@2020}{%
	\textup{2020} Mathematics Subject Classification}
\makeatother

\hypersetup{
  pdftitle={Gradient regularity and potential estimates for fractional drift--diffusion equations in the critical and subcritical ranges},
  pdfauthor={Qi Xue, Chao Zhang},
  pdfsubject={2020 Mathematics Subject Classification: 35R11, 35B65, 47G20},
  pdfkeywords={fractional Laplacian, drift-diffusion equation, Schauder estimate,
    critical drift, gradient regularity, measure data, Riesz potential,
    boundary regularity}
}

\begin{document}

\title[Regularity for fractional drift--diffusion equations]
{Gradient regularity and potential estimates for fractional
drift--diffusion equations in the \rev{critical and subcritical ranges}}

\author[Q. Xue]{Qi Xue}
\address{Qi Xue \hfill\break School of Mathematics, Harbin Institute of Technology,
Harbin 150001, P.R. China}
\email{18332276756@163.com}

\author[C. Zhang]{Chao Zhang$^*$}
\address{Chao Zhang \hfill\break School of Mathematics and Institute for Advanced Study in Mathematics,
Harbin Institute of Technology, Harbin 150001, P.R. China}
\email{czhangmath@hit.edu.cn}
\thanks{$^*$Corresponding author.}

\subjclass[2020]{Primary 35R11; Secondary 35B65, 47G20}
\keywords{fractional Laplacian; drift--diffusion equation; Schauder estimate; gradient regularity; measure data; Riesz potential; boundary regularity}

\begin{abstract}
We establish scale-invariant interior $C^{1,\alpha}$ estimates for bounded
viscosity solutions of
$(-\Delta)^su+b\cdot\nabla u=f$ for $s\in[1/2,1)$ with locally H\"older
$b$ and $f$.  The critical case uses Silvestre's parabolic theorem; the
subcritical case uses Schauder estimates and interpolation.  Applying this
viscosity estimate to drifted Green sections, for finite Radon data above
the critical order we obtain sharp solution and gradient potentials of
orders $2s$ and $2s-1$, together with weak-*--to--strong local $W^{1,1}$
stability; hence the Green-potential SOLA is approximation-independent.
For the normalized whole-space kernels, we identify the classical
second-order limits as $s\uparrow1$, including the logarithmic kernel in
dimension two; at the critical order, the whole-space gradient becomes a
zero-order singular integral.  For zero-exterior problems with compactly
supported drift, we also prove
$u/d^s\in C^{s-\varepsilon}(\overline\Omega)$ and identify the obstruction
when the drift reaches the boundary.
\end{abstract}

\maketitle

\section{Introduction}\label{sec:introduction}

Let $n\ge2$, let $\Omega\subset\R^n$ be open, and consider
\begin{equation}\label{eq:equation}
 (-\Delta)^s u+b(x)\cdot\nabla u=f(x)
 \qquad\text{in }\Omega,
\end{equation}
where $s\in[1/2,1)$.  We study the interior differentiability of bounded
viscosity solutions when $b,f$ are locally H\"older continuous.

The two terms on the left-hand side of \eqref{eq:equation} have orders
$2s$ and $1$, respectively.  Under the change of variables
$u_R(x)=u(x_0+Rx)$, the rescaled drift is
\[
 b_R(x)=R^{2s-1}b(x_0+Rx).
\]
Thus $s>1/2$ is subcritical: the drift becomes a lower-order term at small
scales.
At $s=1/2$, however, the drift and diffusion have the same order and no
small factor is produced by scaling.  This distinction is reflected in the
proof.

Sharp interior Schauder estimates for translation-invariant stable operators
were established by Ros-Oton and Serra~\cite{RosOtonSerra2016}.  In the case
of the fractional Laplacian, bounded exterior data and a $C^\alpha$
right-hand side give interior $C^{2s+\alpha}$ regularity, provided the exponent
is not an integer.  When $s=1/2$, an $L^\infty$ right-hand side gives in
general only $C^{1-\varepsilon}$ regularity.  Thus an endpoint Lipschitz
estimate cannot depend only on the $L^\infty$ norm of the right-hand side.

For equations with drift, Silvestre~\cite{Silvestre2012Diff} proved
differentiability estimates for the critical and supercritical parabolic
equation
\[
 u_t+b\cdot\nabla u+(-\Delta)^s u=f.
\]
At $s=1/2$, spatial $C^\alpha$ regularity of $b$ and $f$ implies spatial
$C^{1,\alpha}$ regularity of bounded solutions.  Since a stationary solution
of \eqref{eq:equation} is a time-independent parabolic solution, this theorem
gives the critical estimate used below.  Related H\"older, probabilistic, and
parabolic Schauder estimates can be found in
\cite{Silvestre2012Holder,Priola2012,Chaudru2020}.  Priola~\cite{Priola2012}
studied SDEs driven by a nondegenerate symmetric stable process using global
resolvent Schauder estimates and an It\^o--Tanaka transformation.  At stable
order one, that argument is based on localization around constant drifts.  In
the present paper, the critical estimate is obtained from Silvestre's PDE
theorem, whereas for $s>1/2$ the drift is absorbed by interpolation as a
lower-order term.

Measure-data problems for nonlocal equations have been studied extensively.
Kuusi, Mingione, and Sire~\cite{KuusiMingioneSire2015} developed existence,
regularity, and nonlinear potential estimates for SOLA.  Kuusi, Nowak, and
Sire~\cite{KuusiNowakSire2022} proved Sobolev and H\"older regularity and
first-order Riesz-potential estimates for linear equations with
H\"older-dependent kernels.  Sharp gradient-potential estimates for nonlinear
nonlocal equations were subsequently obtained by Diening, Kim, Lee, and
Nowak~\cite{DieningKimLeeNowak2025}.  These works treat nonlocal principal
parts without an explicit first-order drift; their gradient estimates are
based on nonlocal excess decay and difference quotients.

For Kato-class drifts, Bogdan and Jakubowski~\cite{BogdanJakubowski2012}
proved comparability of the drifted and drift-free Green functions, while
Chen, Kim, and Song~\cite{ChenKimSong2012} developed the corresponding heat
kernel theory.  These results provide the zeroth-order Green-function bounds
used below.  For measure data, one also needs a viscosity characterization of
Green sections, a derivative estimate in the first Green variable, and
stability of the Green operator under weak-* convergence of measures.

The purpose of this paper is to establish these properties for the
nonsymmetric operator $(-\Delta)^s+b\cdot\nabla$ and to apply them to the
corresponding measure-data problem.  We first prove a scale-invariant interior
$C^{1,\alpha}$ estimate, with explicit dependence on the dimensionless size of
the drift.  At $s=1/2$, the estimate is obtained by applying Silvestre's
parabolic theorem to stationary solutions.  For $s>1/2$, it follows instead
from the drift-free Schauder estimate and an interpolation argument that
absorbs the first-order term.  A regularization compatible with viscosity
convergence then gives the estimate for bounded viscosity solutions of
\eqref{eq:equation}.

We next show that probabilistically harmonic Green sections satisfy the
pasted-test viscosity property needed to apply the interior estimate.  It
follows that
\[
 |\nabla_xG_b^D(x,y)|\lesssim |x-y|^{2s-n-1}.
\]

Using this estimate, we extend the Green operator to finite signed Radon
measures and prove potential bounds of orders $2s$ and $2s-1$ for the solution
and its gradient, respectively.  We also prove sequential continuity from
weak-* convergence of measures to strong convergence in
$W^{1,1}_{\mathrm{loc}}(D)$.  This gives a Green-potential SOLA that is
independent of the approximating sequence and is unique in the class under
consideration.  In addition, we determine the normalized limits of the
whole-space kernels as $s\uparrow1$ and prove boundary regularity when the
drift is supported away from the boundary.

The paper is organized as follows.  The main results are stated in
\cref{sec:main-results}.  In \cref{sec:preliminaries}, we discuss the
solution notion, record scaling, and explain the scope of the weak
formulation.  The analytic ingredients for the interior argument are collected
in \cref{sec:estimates}, and \cref{sec:proof} proves
\cref{thm:main}.  Measure data, Green kernels, potential estimates, and the
limit $s\uparrow1$ are treated in \cref{sec:measure}.  Boundary regularity
and the obstruction created by the drift are discussed in
\cref{sec:boundary}.

\section{Main results}\label{sec:main-results}

For convenience, all principal statements are presented together in this
section.  Their proofs are given in the subsequent sections.  The first
theorem is the interior
Schauder estimate.  The next two theorems identify the sharp potential
orders in the whole space and for the drifted Dirichlet problem.  The fourth
theorem records the local limit as $s\uparrow1$, and the last theorem gives
the boundary consequence when the drift is supported away from the
boundary.  \rev{The interior estimate is proved first and is then applied to
the Green kernel after the required viscosity characterization has been
established.}

Throughout, $C$ denotes a positive constant which may change from line to
line, and subscripts record additional dependence.  We write $A\lesssim B$
for $A\le CB$.  If $\sigma>0$ is not an integer and
$k=\lfloor\sigma\rfloor$, we use the standard convention
$C^\sigma=C^{k,\sigma-k}$.  All H\"older norms are full norms unless a
seminorm is displayed explicitly.
For later localization arguments we use
\begin{equation}\label{eq:tail-definition}
 \Tail_s(v;x_0,R):=R^{2s}
 \int_{\R^n\setminus B_R(x_0)}
 \frac{|v(y)|}{|x_0-y|^{n+2s}}\,\differential y.
\end{equation}

For the measure-data statements, if $0<\beta<n$, $x\in\R^n$, and
$R\in(0,\infty)$, we use the augmented truncated Riesz potential
\begin{equation}\label{eq:riesz-potential-main}
 \I_\beta^{|\mu|}(x,R)
 :=\int_0^R
 \frac{|\mu|(B_\rho(x))}{\rho^{n-\beta}}\,\frac{\differential\rho}{\rho}
 +\frac{|\mu|(B_R(x))}{R^{n-\beta}}.
\end{equation}
For the untruncated potential we set
\[
 \I_\beta^{|\mu|}(x)
 :=\int_0^\infty
 \frac{|\mu|(B_\rho(x))}{\rho^{n-\beta}}
 \frac{\differential\rho}{\rho}
\]
whenever this integral is finite.  The terminal term in the truncated
definition is included to retain equivalence with the kernel
form on $B_R(x)$.

Whenever a fundamental solution is used, $(-\Delta)^s$ is normalized to
have Fourier symbol $|\xi|^{2s}$, and we set
\begin{equation}\label{eq:fundamental-normalized}
 \Phi_{n,s}(x)=\kappa_{n,s}|x|^{2s-n},\qquad
 \kappa_{n,s}:=
 \frac{\Gamma\!\left(\frac{n-2s}{2}\right)}
 {2^{2s}\pi^{n/2}\Gamma(s)}.
\end{equation}

\begin{theorem}[interior gradient Schauder estimate]\label{thm:main}
Let $n\ge2$, $s\in[1/2,1)$ and $\gamma\in(0,1)$.  Let
$B_{2R}(x_0)\Subset\Omega$, with $0<R\le1$, and suppose that
\[
 u\in L^\infty(\R^n)\cap C(\Omega)
\]
is a viscosity solution of \eqref{eq:equation} in $B_{2R}(x_0)$.
Assume that
\[
 b,f\in C^\gamma(B_{2R}(x_0)).
\]
Then $u\in C^{1,\alpha}(B_{R/2}(x_0))$ for every
$\alpha\in(0,\gamma)$.  Moreover, the following estimate holds:
\begin{align}\label{eq:main-estimate}
 &\|\nabla u\|_{L^\infty(B_{R/2}(x_0))}
 +R^\alpha[\nabla u]_{C^\alpha(B_{R/2}(x_0))}\notag\\
 &\quad\le \frac{C}{R}\bigl(
 \|u\|_{L^\infty(\R^n)}
 +R^{2s}\|f\|_{L^\infty(B_{2R}(x_0))}
 +R^{2s+\gamma}[f]_{C^\gamma(B_{2R}(x_0))}\bigr),
\end{align}
where $C$ depends only on $n,s,\gamma,\alpha$ and, monotonically, on the
dimensionless drift size
\begin{equation}\label{eq:scaled-b}
 \mathcal B_R:=R^{2s-1}\|b\|_{L^\infty(B_{2R}(x_0))}
 +R^{2s+\gamma-1}[b]_{C^\gamma(B_{2R}(x_0))}.
\end{equation}
In particular, an admissible exponent $\alpha>0$ always exists.
\end{theorem}

\begin{remark}[drift dependence and endpoint exponents]
The dependence on the dimensionless drift size $\mathcal B_R$ in
\eqref{eq:scaled-b} is essential.  At the critical endpoint $s=1/2$, the
H\"older modulus of $b$ is part of the principal-order information and
cannot be discarded.  We do not claim that the exponent in
\cref{thm:main} is optimal.  The theorem is stated for $\alpha<\gamma$ in
order to avoid endpoint and integer H\"older spaces.  A sharper endpoint
formulation would naturally use the corresponding Zygmund scale.
\end{remark}

\begin{theorem}[whole-space fractional Poisson equation]
\label{thm:whole-space-potential}
Let $n\ge2$, $1/2<s<1$, and let $\mu\in\M(\R^n)$ be a compactly
supported finite signed Radon measure.  With the standard normalization,
define
\begin{equation}\label{eq:whole-space-solution}
 u(x)=\int_{\R^n}\Phi_{n,s}(x-y)\,\differential\mu(y),
 \qquad \Phi_{n,s}(z)=\kappa_{n,s}|z|^{2s-n}.
\end{equation}
Then $u\in L^1_{\mathrm{loc}}(\R^n)$ and $(-\Delta)^su=\mu$ in
$\mathcal D'(\R^n)$.  Each of the following estimates holds at every point
where its right-hand side is finite:
\begin{equation}\label{eq:whole-space-potential-estimates}
 |u(x)|\le C\I_{2s}^{|\mu|}(x),\qquad
 |\nabla u(x)|\le C\I_{2s-1}^{|\mu|}(x).
\end{equation}
The gradient in the second estimate is defined by the absolutely convergent
differentiated potential.  Both orders are sharp in the local scaling sense,
as shown by taking $\mu=\delta_0$.
\end{theorem}

\begin{definition}[\rev{Green-potential SOLA for the measure-data equation}]
\label{def:sola-measure}
Let $D\subset\R^n$ be bounded and let $\mu\in\M(D)$.  Call a sequence
$h_j\in C_c^\infty(D)$ an admissible approximation
of $\mu$ if
\[
 h_j\,\differential x\stackrel{*}{\rightharpoonup}\mu
 \quad\text{in }\M(\overline D),
 \qquad \sup_j\|h_j\|_{L^1(D)}<\infty,
\]
where measures on $D$ are extended by zero to $\partial D$.  A
function $u\in L^1(D)$, extended by zero to $\R^n\setminus D$, is a
\emph{Green-potential solution obtained as a limit of approximations}
(Green-potential SOLA) of
\begin{equation}\label{eq:measure-dirichlet-main}
 \begin{cases}
 (-\Delta)^su+b\cdot\nabla u=\mu&\text{in }D,\\
 u=0&\text{in }\R^n\setminus D,
 \end{cases}
\end{equation}
if there is an admissible approximation for which the Green potentials
\[
 u_j(x):=\int_DG_b^D(x,y)h_j(y)\,\differential y
\]
satisfy $u_j\to u$ in $L^1_{\mathrm{loc}}(D)$.  \rev{The Green kernel and
its zero-exterior operator are specified in
\cref{prop:green-framework}.}
The gradient is not included in this definition; its existence and its
estimate are conclusions of the next theorem.
\end{definition}

\begin{theorem}[\rev{Green operator on measures: stability and sharp bounds}]
\label{thm:measure-potential}
Let $n\ge2$, $1/2<s<1$, $\gamma\in(0,1)$, and let
$D\subset\R^n$ be a bounded $C^{1,1}$ domain.  Assume that
$b\in\mathcal K_n^{2s-1}\cap C^\gamma_{\mathrm{loc}}(D)$, where the Kato
class $\mathcal K_n^{2s-1}$ is defined in
\eqref{eq:kato-measure-section}.  Let
$\mu\in\M(D)$ be a finite signed Radon measure.
\begingroup
\rev{Then the Green
operator extends to a positive linear map}
\[
 \mathcal G_b:\M(D)\longrightarrow
 L^1(D)\cap W^{1,1}_{\mathrm{loc}}(D),
\]
\rev{given by the Green potential}
\begin{equation}\label{eq:green-potential-solution}
 u(x)=\mathcal G_b\mu(x):=\int_DG_b^D(x,y)\,\differential\mu(y)
 \quad\text{for a.e. }x\in D.
\end{equation}
\rev{More precisely, the following conclusions hold.}
\begin{enumerate}
\item \rev{For every $K\Subset D$,}
\begin{equation}\label{eq:green-measure-operator-bound}
 \|\mathcal G_b\mu\|_{L^1(D)}
 +\|\nabla\mathcal G_b\mu\|_{L^1(K)}
 \le C_K|\mu|(D).
\end{equation}
\rev{If $\mu\ge0$, then $\mathcal G_b\mu\ge0$; hence the operator is
order preserving.}

\item \rev{The function $u=\mathcal G_b\mu$ solves
\eqref{eq:measure-dirichlet-main} distributionally:}
\begin{equation}\label{eq:measure-distributional}
 \int_D u(-\Delta)^s\varphi\,\differential x
 +\int_D b\cdot\nabla u\,\varphi\,\differential x
 =\int_D\varphi\,\differential\mu
\end{equation}
for every $\varphi\in C_c^\infty(D)$.

\item \rev{At every $x\in D$ for which the right-hand side below is
finite, the integral in \eqref{eq:green-potential-solution} is absolutely
convergent and}
\begin{equation}\label{eq:u-potential-bound}
 |u(x)|\le C\I_{2s}^{|\mu|}(x,\operatorname{diam}D).
\end{equation}
\rev{For every $K\Subset D$, the weak gradient has the representative}
\begin{equation}\label{eq:green-potential-gradient}
 V_\mu(x):=\int_D\nabla_xG_b^D(x,y)\,\differential\mu(y),
\end{equation}
\rev{defined at every $x\in K$ for which the potential below is finite,
and}
\begin{equation}\label{eq:gradient-potential-bound}
 |V_\mu(x)|\le C_K\I_{2s-1}^{|\mu|}
 (x,2\operatorname{diam}D).
\end{equation}
\rev{In particular, $V_\mu=\nabla u$ almost everywhere on $K$.}

\item \rev{The extended Green operator is sequentially weak-*--to--strong
continuous on bounded subsets of measures: if
$\mu_j\stackrel{*}{\rightharpoonup}\mu$ in $\M(\overline D)$ and
$\sup_j|\mu_j|(D)<\infty$, and neither $\mu$ nor any $\mu_j$ charges
$\partial D$, then}
\begin{equation}\label{eq:green-measure-stability}
 \mathcal G_b\mu_j\longrightarrow\mathcal G_b\mu
 \quad\text{in }W^{1,1}(K)
 \qquad\text{for every }K\Subset D.
\end{equation}
\rev{Consequently, $u$ is the unique Green-potential SOLA of
\eqref{eq:measure-dirichlet-main}, and every admissible approximation in
\cref{def:sola-measure} converges to $u$ in
$W^{1,1}_{\mathrm{loc}}(D)$.}
\end{enumerate}
\rev{The potential orders $2s$ and $2s-1$ are optimal over the class of
problems covered by the theorem.  The constant in
\eqref{eq:u-potential-bound} depends on $n,s,D$ and the Kato modulus of
$b$; $C_K$ also depends on $K$, $\gamma$, and the local $C^\gamma$ norm of
$b$.}  The constants are for fixed $s$; no uniformity as $s\uparrow1$ is
asserted.
\endgroup
\end{theorem}

\begin{remark}[scope of the measure-data theorem]
\label{rem:measure-scope}
\rev{No separate compatibility assumption on the Green sections is made in
\cref{thm:measure-potential}.  Their probabilistic harmonicity, supplied by
the Green-kernel construction, is converted in
\cref{prop:green-viscosity} into the pasted-test viscosity property needed
for the interior estimate.  Continuity off the diagonal and Green-function
comparison follow from \cite[Lemma~7 and Theorem~1]{BogdanJakubowski2012}.
The existence, continuity, and pointwise bound of $\nabla_xG_b^D$ are then
conclusions of \cref{lem:green-gradient}.  Conclusion~(4) is a stability
property of the Green operator, not an additional condition on the
approximating sequence.}
\end{remark}

\begin{theorem}[consistency as $s\uparrow1$]
\label{thm:s-to-one}
With the normalization in \eqref{eq:fundamental-normalized}, the following
statements hold:
\begin{enumerate}
\item If $n\ge3$, then locally uniformly on $\R^n\setminus\{0\}$,
\begin{equation}\label{eq:newton-limit}
 \Phi_{n,s}(x)\longrightarrow
 \frac{1}{(n-2)|S^{n-1}|}|x|^{2-n}.
\end{equation}
For compactly supported $\mu$, the associated potentials and gradients
therefore converge locally uniformly outside $\operatorname{supp}\mu$.
At a point $x$ in the support, the same conclusions hold if, for some
$\delta>0$,
\begin{equation}\label{eq:s-limit-dominating-potentials}
 \I_{2-\delta}^{|\mu|}(x)<\infty,
 \qquad \I_{1-\delta}^{|\mu|}(x)<\infty,
\end{equation}
respectively.
\item If $n=2$, then locally uniformly on $\R^2\setminus\{0\}$,
\begin{equation}\label{eq:log-limit}
 \Phi_{2,s}(x)-\frac{1}{4\pi(1-s)}
 \longrightarrow -\frac{1}{2\pi}\log|x|+c_0,
\end{equation}
while
\begin{equation}\label{eq:gradient-log-limit}
 \nabla\Phi_{2,s}(x)\longrightarrow
 -\frac{1}{2\pi}\frac{x}{|x|^2}.
\end{equation}
Thus the gradient estimate converges to an $\I_1$ estimate, whereas the
solution requires a logarithmic potential after fixing an additive
normalization.
\item The weak gradient exponent satisfies
\begin{equation}\label{eq:weak-exponent-limit}
 \frac{n}{n-2s+1}\longrightarrow\frac{n}{n-1}.
\end{equation}
\end{enumerate}
These statements concern normalized kernels, potential orders, and mapping
exponents; they do not assert uniform boundedness of the constants in the
estimates as $s\uparrow1$.
\end{theorem}

\begin{theorem}[boundary regularity with interior drift]\label{thm:boundary-drift}
Let $s\in[1/2,1)$, $\gamma\in(0,1)$, and let $\Omega\subset\R^n$ be a
bounded $C^{1,1}$ domain.  Suppose that
\[
 b,f\in C^\gamma(\overline\Omega),
 \qquad K:=\operatorname{supp}_{\Omega}b\Subset\Omega.
\]
Let $u\in L^\infty(\R^n)\cap C(\Omega)$ be a viscosity solution of
the problem
\begin{equation}\label{eq:dirichlet-drift}
 \begin{cases}
 (-\Delta)^su+b\cdot\nabla u=f&\text{in }\Omega,\\
 u=0&\text{in }\R^n\setminus\Omega.
 \end{cases}
\end{equation}
Let $d(x)=\operatorname{dist}(x,\R^n\setminus\Omega)$ in $\Omega$, extended
by zero to $\R^n\setminus\Omega$.  Then, for every
$\varepsilon\in(0,s)$,
\begin{equation}\label{eq:boundary-drift-estimate}
 u\in C^s(\R^n),
 \qquad \frac{u}{d^s}\in
 C^{s-\varepsilon}(\overline\Omega),
\end{equation}
and
\begin{align}\label{eq:boundary-drift-norm}
 \|u\|_{C^s(\R^n)}
 +\left\|\frac{u}{d^s}\right\|_{C^{s-\varepsilon}(\overline\Omega)}\le C\bigl(
 \|u\|_{L^\infty(\R^n)}+
 \|f\|_{C^\gamma(\Omega)}\bigr).
\end{align}
Here $C$ depends on $n,s,\gamma,\varepsilon,\Omega$,
$\operatorname{dist}(K,\partial\Omega)$ and
$\|b\|_{C^\gamma(\Omega)}$.
\end{theorem}

The first theorem is proved in \cref{sec:proof}; the whole-space and
measure-data estimates, together with the $s\uparrow1$ limit, are proved in
\cref{sec:measure}; and the last theorem is proved in \cref{sec:boundary}.

\rev{The interior estimate does not use a Green-function representation.
Green-kernel hypotheses enter only in \cref{sec:measure}, where the interior
estimate is applied away from the pole.}

\section{Viscosity solutions, notation, and scaling}
\label{sec:preliminaries}

\rev{We begin by fixing the normalization and the viscosity and energy
formulations, and then record the scaling and regularization used in the
interior argument.  The same pasted-test framework will later be used for
the Green sections; the energy formulation is needed only for
\cref{cor:weak-viscosity-equivalence}.}

We normalize the fractional Laplacian by
\begin{equation}\label{eq:frac-lap}
 (-\Delta)^s u(x)=c_{n,s}\,\operatorname{P.V.}
 \int_{\R^n}\frac{u(x)-u(y)}{|x-y|^{n+2s}}\,\differential y.
\end{equation}
For estimates at a fixed value of $s$, the precise value of the positive
constant $c_{n,s}$ is immaterial.  Whenever the limit $s\uparrow1$ is
discussed, however, we use the standard normalization for which the Fourier
symbol of $(-\Delta)^s$ is $|\xi|^{2s}$.  With this normalization,
$(-\Delta)^s\varphi\to-\Delta\varphi$ for every Schwartz function
$\varphi$.

For later reference, define
\[
 L^1_{2s}(\R^n)=\left\{v\in L^1_{\mathrm{loc}}(\R^n):
 \int_{\R^n}\frac{|v(y)|}{1+|y|^{n+2s}}\,\differential y<\infty\right\}.
\]
The assumption $u\in L^\infty(\R^n)$ implies $u\in L^1_{2s}(\R^n)$.

If $u\in L^1_{2s}(\R^n)$, $\phi\in C^2(B_\rho(x_0))$, and
$0<r<\rho$, define the function obtained by pasting the test function to
the original solution by
\begin{equation}\label{eq:pasted-test}
 \phi^{\,r,u}_{x_0}(x):=
 \begin{cases}
  \phi(x),&x\in B_r(x_0),\\
  u(x),&x\in\R^n\setminus B_r(x_0).
 \end{cases}
\end{equation}
Its fractional Laplacian at the contact point is evaluated directly from
the original principal-value definition:
\begin{equation}\label{eq:pasted-fractional-laplacian}
 (-\Delta)^s\phi^{\,r,u}_{x_0}(x_0)
 :=c_{n,s}\operatorname{P.V.}\int_{\R^n}
 \frac{\phi^{\,r,u}_{x_0}(x_0)-\phi^{\,r,u}_{x_0}(y)}
 {|x_0-y|^{n+2s}}\,\differential y.
\end{equation}
The principal value is finite: in $B_r(x_0)$ the first-order Taylor term
of the $C^2$ function $\phi$ has zero principal value against the symmetric
kernel, and the remainder is $O(|y-x_0|^2)$; outside the ball, finiteness
follows from $u\in L^1_{2s}(\R^n)$.

\begin{definition}[viscosity solutions]\label{def:viscosity}
A function $u:\R^n\to\R$ belonging to $L^1_{2s}(\R^n)$ and upper
semicontinuous in $\Omega$ is a viscosity
subsolution of \eqref{eq:equation} if the following holds.  Whenever
$x_0\in\Omega$, $B_\rho(x_0)\Subset\Omega$, and
$\phi\in C^2(B_\rho(x_0))$ satisfy
\[
 u(x_0)=\phi(x_0),\qquad
 u-\phi\le0\quad\text{in }B_\rho(x_0),
\]
one has, for every $0<r<\rho$,
\begin{equation}\label{eq:viscosity-sub}
 (-\Delta)^s\phi^{\,r,u}_{x_0}(x_0)
 +b(x_0)\cdot\nabla\phi(x_0)\le f(x_0).
\end{equation}
A function in $L^1_{2s}(\R^n)$ that is lower semicontinuous in $\Omega$ is
a viscosity supersolution if, whenever a $C^2$ function touches it from
below as above, the reverse inequality holds in
\eqref{eq:viscosity-sub}.  A function in $L^1_{2s}(\R^n)$ that is continuous
in $\Omega$ is a viscosity solution if it is both a viscosity subsolution
and a viscosity supersolution.
\end{definition}

Thus the test function controls the singular part of the original integral,
whereas the solution itself is retained in the nonlocal tail.  This is the
standard pasted-test formulation used for nonlocal viscosity solutions and
is equivalent to the global-test formulation in
\cite[Section~3]{Silvestre2012Diff}.  No second-difference notation is part
of the definition.

For comparison with the variational formulation, we also record the local
energy notion.  It will be used only in the equivalence result below; the
main Schauder theorem itself remains a viscosity statement.

\begin{definition}[local energy weak solution]\label{def:weak}
Let $U\Subset\Omega$ and assume that $\divergence b=0$ in $U$ in the
distributional sense.  A function
$u\in H^{s}_{\mathrm{loc}}(U)\cap L^1_{2s}(\R^n)$ is a local energy weak
solution of \eqref{eq:equation} in $U$ if
\begin{align}\label{eq:weak-form}
 &\frac{c_{n,s}}2\iint_{\R^n\times\R^n}
 \frac{(u(x)-u(y))(\varphi(x)-\varphi(y))}{|x-y|^{n+2s}}
 \,\differential x\differential y\notag\\
 &\qquad-\int_U u\,b\cdot\nabla\varphi\,\differential x
 =\int_U f\varphi\,\differential x
\end{align}
for every $\varphi\in C_c^\infty(U)$.
\end{definition}

Throughout, if $D\subset\R^n$ is open, $H_0^s(D)$ denotes the closure of
$C_c^\infty(D)$ in $H^s(\R^n)$; its elements are identified with their zero
extensions outside $D$.  This convention is used also at $s=1/2$.

\begin{lemma}[the drift form on the energy space]
\label{lem:drift-form}
Let $s\in[1/2,1)$, let $D$ be a ball, and assume
$b\in W^{1,\infty}(D;\R^n)$ and $\divergence b=0$.  The form initially
defined on smooth compactly supported functions by
\[
 \mathfrak d_b(v,\varphi):=-\int_D v\,b\cdot\nabla\varphi
\]
extends continuously to $H_0^s(D)\times H_0^s(D)$.  Moreover,
\begin{equation}\label{eq:drift-cancellation}
 \mathfrak d_b(v,v)=0\qquad\text{for every }v\in H_0^s(D).
\end{equation}
\end{lemma}

\begin{proof}
Multiplication by a $W^{1,\infty}$ function is bounded on $H^s$.  Also
$\nabla:H^s\to H^{s-1}$ is bounded, and $H^{s-1}\hookrightarrow H^{-s}$
because $s\ge1/2$.  Consequently,
\[
 |\mathfrak d_b(v,\varphi)|
 =|\langle v,b\cdot\nabla\varphi\rangle_{H^s,H^{-s}}|
 \le C\|b\|_{W^{1,\infty}(D)}
       \|v\|_{H^s(\R^n)}\|\varphi\|_{H^s(\R^n)}.
\]
This proves the extension.  If $v\in C_c^\infty(D)$, integration by parts
and $\divergence b=0$ give
\[
 \mathfrak d_b(v,v)=-\frac12\int_D b\cdot\nabla(v^2)
 =\frac12\int_D(\divergence b)v^2=0.
\]
For a general $v\in H_0^s(D)$ choose $v_k\in C_c^\infty(D)$ with
$v_k\to v$ in $H^s$.  Continuity of the extended form and the preceding
identity imply
$\mathfrak d_b(v,v)=\lim_k\mathfrak d_b(v_k,v_k)=0$.
\end{proof}

\begin{definition}[regularized local approximation]
\label{def:approximable-viscosity}
Let $B_{2R}(x_0)\Subset\Omega$.  A sequence $(u_j,b_j,f_j)$ is called a
\emph{regularized local approximation} of a bounded solution $u$ in
$B_{2R}(x_0)$ if,
for every $R'<2R$, there are
\[
 b_j,f_j\in C^\infty(B_{R'}(x_0)),\qquad
 u_j\in C^\infty(B_{R'}(x_0))\cap L^\infty(\R^n),
\]
such that
\begin{enumerate}
\item $(-\Delta)^su_j+b_j\cdot\nabla u_j=f_j$ pointwise in
$B_{R'}(x_0)$;
\item $b_j\to b$ and $f_j\to f$ in $C^\alpha$ on compact subsets for
all $\alpha<\gamma$;
\item $u_j\to u$ locally uniformly in $B_{R'}(x_0)$ and
$\sup_j\|u_j\|_{L^\infty(\R^n)}<\infty$;
\item on each compactly contained ball, the $C^\gamma$ norms of $b_j$ and
$f_j$ are bounded by a fixed multiple of the corresponding norms of $b$
and $f$.
\end{enumerate}
The four conditions isolate the compactness needed to pass the classical
a priori estimate to a nonsmooth viscosity solution.  The next lemma shows
that every bounded viscosity solution considered here admits such an
approximation.
\end{definition}

\rev{The proof uses two localization tools: exact scaling and
viscosity-stable regularization.}

The following scaling calculation will be used repeatedly.

\begin{lemma}[scaling]\label{lem:scaling}
Let $u$ solve \eqref{eq:equation} in $B_{2R}(x_0)$ and set
\begin{equation}\label{eq:rescaling}
 U(x)=u(x_0+Rx),\quad
 B(x)=R^{2s-1}b(x_0+Rx),\quad
 F(x)=R^{2s}f(x_0+Rx).
\end{equation}
Then
\[
 (-\Delta)^sU+B\cdot\nabla U=F\qquad\text{in }B_2.
\]
Furthermore,
\begin{align*}
 \|B\|_{L^\infty(B_2)}
 &=R^{2s-1}\|b\|_{L^\infty(B_{2R}(x_0))},\\
 [B]_{C^\gamma(B_2)}
 &=R^{2s+\gamma-1}[b]_{C^\gamma(B_{2R}(x_0))},\\
 \|F\|_{L^\infty(B_2)}
 &=R^{2s}\|f\|_{L^\infty(B_{2R}(x_0))},\\
 [F]_{C^\gamma(B_2)}
 &=R^{2s+\gamma}[f]_{C^\gamma(B_{2R}(x_0))}.
\end{align*}
\end{lemma}

\rev{The next lemma preserves the global tail bound and the local
H\"older norms under regularization.}

\begin{lemma}[canonical local approximation]
\label{lem:canonical-approximation}
Under the assumptions of \cref{thm:main}, every bounded viscosity solution
admits a regularized local approximation in the sense of
\cref{def:approximable-viscosity}.
\end{lemma}

\begin{proof}
\begingroup
Fix $R'<2R$ and an intermediate ball
$B_{R'}(x_0)\Subset B_{R''}(x_0)\Subset B_{2R}(x_0)$.  Choose smooth
$b_j,f_j$ by mollifying on a slightly larger ball.  Then $b_j\to b$ and
$f_j\to f$ in $C^\alpha$ on compact subsets for every $\alpha<\gamma$,
while their $C^\gamma$ norms remain bounded by a fixed multiple of the
corresponding norms of $b$ and $f$.  By truncation followed by
mollification, choose $g_j\in C_c^\infty(\R^n)$
such that
\[
 \begin{aligned}
 g_j&\to u &&\text{in }L^1_{2s}(\R^n)\text{ and locally uniformly in }
 B_{R''}(x_0),\\
 \sup_j\|g_j\|_{L^\infty(\R^n)}&\le
 \|u\|_{L^\infty(\R^n)}+1.
 \end{aligned}
\]
Perron's method, comparison in the bounded class, and stability under
suprema \cite[Theorem~2]{BarlesImbert2008} give a solution of
\[
 \begin{cases}
 (-\Delta)^su_j+b_j\cdot\nabla u_j=f_j&\text{in }B_{R''}(x_0),\\
 u_j=g_j&\text{in }\R^n\setminus B_{R''}(x_0).
 \end{cases}
\]
We first justify that elliptic regularity applies to this Perron solution.
For a linear equation with smooth coefficients, the viscosity inequalities
imply the corresponding distributional inequalities: on a compactly
contained ball, apply the usual sup- and inf-convolutions, test the resulting
almost-everywhere inequalities against nonnegative smooth functions, and let
the convolution parameter tend to zero.  Smoothness of $b_j$ treats the local
term, while the uniform bound controls the nonlocal tail.  Thus $u_j$ solves
the equation distributionally.  Interior smoothness now follows from
elliptic pseudodifferential regularity.  Indeed, the drift is lower order when
$s>1/2$; when $s=1/2$, the principal symbol
$|\xi|+i b_j(x)\cdot\xi$ has modulus at least $|\xi|$.  Hence
$u_j\in C^\infty(B_{R'}(x_0))$.  The maximum principle gives
\[
 \sup_j\|u_j\|_{L^\infty(\R^n)}
 \le C\bigl(\|u\|_{L^\infty(\R^n)}+
             \|f\|_{L^\infty(B_{R''}(x_0))}\bigr).
\]
By \cite[Theorem~1]{BarlesImbert2008}, the upper and lower half-relaxed
limits are a subsolution and a supersolution of the limiting equation.  To
make the nonlocal passage explicit, split the pasted-test integral at a fixed
small radius.  The smooth test function handles the singular part; the
uniform $L^\infty$ bound gives a common integrable majorant for the far part;
and $g_j\to u$ in $L^1_{2s}(\R^n)$ identifies the exterior datum.  Thus the
hypotheses of the stability theorem are satisfied.  Comparison with the
original solution identifies both limits with $u$, and hence
$u_j\to u$ locally uniformly in $B_{R''}(x_0)$.
A diagonal choice as $R'\uparrow2R$ gives a sequence satisfying
\cref{def:approximable-viscosity}.
\endgroup
\end{proof}

\rev{It remains, in the subcritical case, to estimate smooth unit-scale
solutions in terms of the norms preserved by this approximation.}

\section{Analytic ingredients for the interior estimate}\label{sec:estimates}

\rev{The proof uses the drift-free Schauder estimate, Silvestre's critical
estimate, and, for $s>1/2$, interpolation followed by a nested-ball
iteration.  The interpolation step fails at $s=1/2$, where the drift and
diffusion have the same order.}

We use the following fractional Schauder estimate, which is the
fractional-Laplacian case of
\cite[Corollary~3.5]{RosOtonSerra2016}.

\begin{lemma}[fractional Schauder estimate]\label{lem:fractional-schauder}
Let $0<s<1$, $0<\alpha<1$, and $2s+\alpha\notin\mathbb N$.  If
$v\in L^\infty(\R^n)$ satisfies
\[
 (-\Delta)^s v=g\qquad\text{in }B_1,
 \qquad g\in C^\alpha(B_1),
\]
then
\begin{equation}\label{eq:pure-schauder}
 \|v\|_{C^{2s+\alpha}(B_{1/2})}
 \le C\bigl(\|v\|_{L^\infty(\R^n)}
             +\|g\|_{C^\alpha(B_1)}\bigr),
\end{equation}
where $C=C(n,s,\alpha)$.
\end{lemma}

We will use its localized, nested-ball consequence.  If
$B_r\Subset B_{r'}\Subset B_1$, multiplication by a cutoff equal to one on
$B_{r'}$ and application of \cref{lem:fractional-schauder} give
\begin{equation}\label{eq:localized-schauder}
 \|v\|_{C^{2s+\alpha}(B_r)}
 \le C_{r'-r}\bigl(
 \|v\|_{L^\infty(\R^n)}
 +\|(-\Delta)^sv\|_{C^\alpha(B_{r'})}\bigr).
\end{equation}
\rev{Covering $B_r$ by balls of radius comparable to $r'-r$, followed by
rescaling and cutoff, gives \eqref{eq:localized-schauder}.  The cutoff
commutator is supported a distance comparable to $r'-r$ from the diagonal;
differentiating its kernel yields
$C_{r'-r}\le C(r'-r)^{-m}$ for some structural $m>0$.}

At the critical endpoint we use the following consequence of Silvestre's
parabolic differentiability theorem.

\begin{lemma}[critical drift estimate]\label{lem:critical}
Let $0<\alpha<1$.  Suppose that $v$ is a bounded viscosity solution of
\begin{equation}\label{eq:critical-equation}
 (-\Delta)^{1/2}v+c(x)\cdot\nabla v=g
 \qquad\text{in }B_1,
\end{equation}
where $c,g\in C^\alpha(B_1)$.  Then
\begin{equation}\label{eq:critical-estimate}
 \|v\|_{C^{1,\alpha}(B_{1/2})}
 \le C\bigl(\|v\|_{L^\infty(\R^n)}
             +\|g\|_{C^\alpha(B_1)}\bigr),
\end{equation}
where $C=C(n,\alpha,\|c\|_{C^\alpha(B_1)})$.
\end{lemma}

\begin{proof}
Define time-independent functions on $[-1,0]\times\R^n$ by
\[
 V(t,x)=v(x),\qquad C(t,x)=c(x),\qquad G(t,x)=g(x).
\]
Then
\[
 V_t+C\cdot\nabla V+(-\Delta)^{1/2}V=G
 \qquad\text{in }[-1,0]\times B_1.
\]
Apply \cite[Theorem~1.1]{Silvestre2012Diff}.  In the notation of that
theorem the coefficient exponent is $1-2s+\alpha$, which equals $\alpha$
when $s=1/2$, and its restriction $0<\alpha<2s$ becomes $0<\alpha<1$.
The theorem yields \eqref{eq:critical-estimate}.  Its proof applies to
viscosity solutions by \cite[Section~3]{Silvestre2012Diff}.
\end{proof}

We also use the standard interpolation inequality in H\"older spaces.

\begin{lemma}[interpolation]\label{lem:interpolation}
Let $s>1/2$, $0<\alpha<1$, and
$B_r\Subset B_{r'}$.  For every $\varepsilon>0$,
\begin{equation}\label{eq:interpolation}
 \|v\|_{C^{1+\alpha}(B_r)}
 \le \varepsilon\|v\|_{C^{2s+\alpha}(B_{r'})}
 +C_{\varepsilon,r'-r}\|v\|_{L^\infty(B_{r'})}.
\end{equation}
\end{lemma}

\begin{proof}
\begingroup
Set $\sigma=2s+\alpha$ and $\tau=1+\alpha$, so that
$0<\tau<\sigma$.  After rescaling by the gap $r'-r$, the standard H\"older
interpolation inequality on concentric unit balls gives
\[
 \|v\|_{C^\tau(B_r)}
 \le C\|v\|_{L^\infty(B_{r'})}^{1-\tau/\sigma}
       \|v\|_{C^\sigma(B_{r'})}^{\tau/\sigma}
 +C_{r'-r}\|v\|_{L^\infty(B_{r'})}.
\]
Young's inequality with exponents $\sigma/\tau$ and
$\sigma/(\sigma-\tau)$ turns the mixed term into
$\varepsilon\|v\|_{C^\sigma(B_{r'})}$ plus a multiple of the
$L^\infty$ norm.  This is \eqref{eq:interpolation}, with all dependence on
$\varepsilon$ and on the gap absorbed into $C_{\varepsilon,r'-r}$.
\endgroup
\end{proof}

\begin{lemma}[nested-balls iteration]\label{lem:nested-balls}
Let $a<A$, let $m>0$, and let $X:[a,A)\to[0,\infty)$ be bounded.
Assume that there are constants $0<\theta<1$, $H\ge0$, and
$C_0\ge1$ such that, for all $a\le r<r'<A$,
\[
 X(r)\le \theta X(r')+C_0(r'-r)^{-m}H.
\]
Then
\[
 X(a)\le C(m,\theta)C_0(A-a)^{-m}H.
\]
\end{lemma}

\begin{proof}
Choose $\tau\in(0,1)$ so that $\theta\tau^{-m}<1$, and set
\[
 r_j=a+(A-a)(1-\tau^j),\qquad j=0,1,2,\ldots.
\]
Then $r_0=a$, $r_j\uparrow A$, and
$r_{j+1}-r_j=(A-a)(1-\tau)\tau^j$.  Repeated use of the hypothesis
with $(r,r')=(r_j,r_{j+1})$ gives, for every $N\ge1$,
\[
 X(a)\le \theta^NX(r_N)+C_0H\sum_{j=0}^{N-1}
 \theta^j\bigl((A-a)(1-\tau)\tau^j\bigr)^{-m}.
\]
The sum is at most
\[
 (A-a)^{-m}(1-\tau)^{-m}
 \sum_{j=0}^{\infty}(\theta\tau^{-m})^j,
\]
which is finite.  Since $X$ is bounded, the remainder
$\theta^NX(r_N)$ tends to zero.  Letting $N\to\infty$ proves the
assertion.
\end{proof}

\section{Proof of the main theorem}\label{sec:proof}

\rev{After scaling, the critical case follows from
\cref{lem:critical}.  For $s>1/2$, we prove the estimate first for smooth
solutions and then use \cref{lem:canonical-approximation}.}

\begin{proof}[Proof of \cref{thm:main}]
Apply the scaling in \cref{lem:scaling}.  It suffices to prove
\begin{equation}\label{eq:unit-estimate}
 \|U\|_{C^{1,\alpha}(B_{1/2})}
 \le C\bigl(\|U\|_{L^\infty(\R^n)}+\|F\|_{C^\gamma(B_2)}\bigr),
\end{equation}
where
\[
 (-\Delta)^sU+B\cdot\nabla U=F\qquad\text{in }B_2
\]
and $C$ has the asserted dependence on $B$.

\smallskip
\noindent\emph{Critical case $s=1/2$.}
Since $U$ is a viscosity solution and $C^\gamma(B_2)$ embeds continuously
into $C^\alpha(B_2)$ for $0<\alpha<\gamma$, all assumptions of
\cref{lem:critical} hold directly on $B_1$: the equation is valid there,
$U$ is bounded on $\R^n$, and $B,F\in C^\alpha(B_1)$.  Therefore
\[
 \|U\|_{C^{1,\alpha}(B_{1/2})}
 \le C\bigl(\|U\|_{L^\infty(\R^n)}
             +\|F\|_{C^\alpha(B_1)}\bigr)
 \le C\bigl(\|U\|_{L^\infty(\R^n)}
             +\|F\|_{C^\gamma(B_2)}\bigr).
\]
This proves \eqref{eq:unit-estimate} at the endpoint without assuming any
prior Lipschitz regularity.

\smallskip
\noindent\emph{Subcritical case $s>1/2$: the a priori estimate.}
Choose once and for all
\[
 \beta\in(\alpha,\gamma),\qquad 2s+\beta\notin\mathbb N.
\]
We first suppose that $U$, $B$, and $F$ are smooth.  For
$1/2\le r<1$, write
\[
 X(r)=\|U\|_{C^{2s+\beta}(B_r)}.
\]
On $B_1$,
\[
 (-\Delta)^sU=F-B\cdot\nabla U.
\]
Because the present a priori argument assumes that $U$ is smooth in $B_2$,
the function $X$ is bounded on $[1/2,1)$; thus the iteration lemma applies
without an implicit limiting argument.
\rev{Fix $1/2\le r<r''<1$ and set $r'=(r+r'')/2$.  Then both gaps
$r'-r$ and $r''-r'$ equal $(r''-r)/2$.  From
\eqref{eq:localized-schauder},}
\begin{equation}\label{eq:subcritical-first}
 X(r)\le C_{r'-r}\bigl(
 \|U\|_{L^\infty(\R^n)}+\|F\|_{C^\beta(B_{r'})}
 +\|B\cdot\nabla U\|_{C^\beta(B_{r'})}\bigr).
\end{equation}
The H\"older product inequality gives
\begin{equation}\label{eq:product}
 \|B\cdot\nabla U\|_{C^\beta(B_{r'})}
 \le C\|B\|_{C^\beta(B_1)}
       \|U\|_{C^{1+\beta}(B_{r'})}.
\end{equation}
By \cref{lem:interpolation}, for every $\varepsilon>0$,
\begin{equation}\label{eq:subcritical-interpolation}
 \|U\|_{C^{1+\beta}(B_{r'})}
 \le \varepsilon X(r'')
 +C_{\varepsilon,r''-r'}\|U\|_{L^\infty(B_{r''})}.
\end{equation}
\rev{Choose $\varepsilon$ so that
\[
 C_{r'-r}\,C\varepsilon\|B\|_{C^\beta(B_1)}\le\frac12.
\]
Both constants depend at most polynomially on the reciprocal gaps.  Since
$r'-r=r''-r'=(r''-r)/2$, combining
\eqref{eq:subcritical-first}--\eqref{eq:subcritical-interpolation} gives a
nested-ball inequality of the form
\begin{equation}\label{eq:nested}
 X(r)\le\frac12X(r'')
 +C(r''-r)^{-m}\bigl(
 \|U\|_{L^\infty(\R^n)}+\|F\|_{C^\beta(B_1)}\bigr).
\end{equation}
Here $m>0$ and $C$ depend only on $n,s,\beta$ and
$\|B\|_{C^\beta(B_1)}$.}
Applying \cref{lem:nested-balls} with $a=1/2$, $A=1$, and
$H=\|U\|_{L^\infty(\R^n)}+\|F\|_{C^\beta(B_1)}$ yields
\begin{equation}\label{eq:subcritical-schauder}
 \|U\|_{C^{2s+\beta}(B_{1/2})}
 \le C\bigl(
 \|U\|_{L^\infty(\R^n)}+\|F\|_{C^\beta(B_1)}\bigr).
\end{equation}
Here $C$ depends on $n,s,\beta$ and
$\|B\|_{C^\beta(B_1)}$.  Since $2s+\beta>1+\beta$, this estimate gives
$U\in C^{1,\beta}(B_{1/2})$ and hence \eqref{eq:unit-estimate}.

\smallskip
\noindent\emph{Passage to viscosity solutions.}
\rev{By \cref{lem:canonical-approximation}, after unit scaling there is a
regularized local approximation $(U_j,B_j,F_j)$ in the sense of
\cref{def:approximable-viscosity}.  The estimate at the auxiliary exponent
$\beta$ is uniform in $j$, since the global $L^\infty$ norms of $U_j$ and
the local coefficient norms are uniformly bounded.  Hence
\[
 \sup_j\|U_j\|_{C^{1,\beta}(B_{1/2})}<\infty.
\]
By Arzel\`a--Ascoli, a subsequence of $U_j$ and $\nabla U_j$ converges
uniformly on compact subsets of $B_{1/2}$.  Since $U_j\to U$ locally
uniformly, the limit of the gradients is $\nabla U$.  For
$x,y\in B_{1/2}$ with $x\ne y$,
\[
 |\nabla U(x)-\nabla U(y)|
 =\lim_{j\to\infty}|\nabla U_j(x)-\nabla U_j(y)|
 \le C|x-y|^\beta.
\]
Therefore $U\in C^{1,\beta}(B_{1/2})$, which implies the asserted
$C^{1,\alpha}$ estimate.}

\rev{Scaling back in \eqref{eq:rescaling},
\[
 \nabla U(x)=R\nabla u(x_0+Rx),\qquad
 [\nabla U]_{C^\alpha(B_{1/2})}
 =R^{1+\alpha}[\nabla u]_{C^\alpha(B_{R/2}(x_0))}
\]
and
\[
 \|F\|_{C^\gamma(B_2)}
 \le R^{2s}\|f\|_{L^\infty(B_{2R}(x_0))}
 +R^{2s+\gamma}[f]_{C^\gamma(B_{2R}(x_0))},
\]
together with \eqref{eq:scaled-b}, give \eqref{eq:main-estimate}.}
\end{proof}

\begin{remark}[regularization]
The main theorem is stated entirely in the viscosity framework.  The
canonical approximation in \cref{lem:canonical-approximation} supplies the
smooth functions needed for the a priori argument; no differentiability or
additional approximation hypothesis is imposed on the original solution.
\end{remark}

For comparison with the drift-free weak-to-viscosity result in
\cite[Theorem~1]{ServadeiValdinoci2014}, we record the corresponding local statement
under a divergence-free Lipschitz drift.

\rev{The additional assumptions below are used only in the energy
argument.  The converse implication is obtained from the common
approximation and energy uniqueness, rather than being attributed to the
drift-free theorem.}

\begin{corollary}[equivalence of the two solution notions]
\label{cor:weak-viscosity-equivalence}
Assume the hypotheses of \cref{thm:main} and, in addition, that
$b\in W^{1,\infty}_{\mathrm{loc}}(B_{2R}(x_0);\R^n)$ and
$\divergence b=0$ in distributions.  Let
\[
 u\in L^\infty(\R^n)\cap C(B_{2R}(x_0))
       \cap \rev{H^s(\R^n)}.
\]
Then $u$ is a local energy weak solution in $B_{2R}(x_0)$ if and only if
it is a viscosity solution there.  Consequently, every such weak solution
satisfies the estimate in \cref{thm:main}.
\end{corollary}

\begin{proof}
Fix concentric balls $D\Subset D'\Subset B_{2R}(x_0)$.  We prove the
equivalence in $D$; since $D$ is arbitrary, this gives the local statement.

\smallskip
\noindent\emph{Step 1: common variational and viscosity approximations.}
\begingroup
Choose a ball $D''$ with
$D'\Subset D''\Subset B_{2R}(x_0)$.  Mollify $b$ and $f$ in $D''$ with
radii smaller than $\operatorname{dist}(D',\partial D'')$.  Their
restrictions $b_j,f_j$ to $D'$ satisfy
\begin{align*}
 b_j&\to b\quad\text{uniformly on }D',
 &\sup_j\|b_j\|_{W^{1,\infty}(D')}&<\infty,\\
 f_j&\to f\quad\text{uniformly on }D',
 &\divergence b_j&=0.
\end{align*}
Convolution preserves $\divergence b_j=0$ on $D'$, so no global
divergence-free extension is needed.  Truncation and mollification give
$g_j\in C_c^\infty(\R^n)$ such that
\[
 g_j\to u\quad\text{in }H^s(\R^n)\cap L^1_{2s}(\R^n)
 \text{ and uniformly near }\overline D,
\]
with a uniform $L^\infty$ bound.  Write
\[
 \mathcal E_s(v,\varphi):=\frac{c_{n,s}}2
 \iint_{\R^n\times\R^n}
 \frac{(v(x)-v(y))(\varphi(x)-\varphi(y))}
 {|x-y|^{n+2s}}\,\differential x\differential y.
\]
For $w,\varphi\in H_0^s(D)$ set
\[
 \mathcal A_j(w,\varphi)
 :=\mathcal E_s(w,\varphi)+\mathfrak d_{b_j}(w,\varphi).
\]
By \cref{lem:drift-form}, $\mathcal A_j$ is continuous, and
$\mathcal A_j(w,w)=\mathcal E_s(w,w)$.  The fractional Poincar\'e
inequality therefore makes $\mathcal A_j$ coercive on $H_0^s(D)$.
Lax--Milgram gives a unique $w_j\in H_0^s(D)$ satisfying
\begin{equation}\label{eq:variational-approximation}
 \mathcal A_j(w_j,\varphi)
 =\int_Df_j\varphi\,\differential x
  -\mathcal E_s(g_j,\varphi)-\mathfrak d_{b_j}(g_j,\varphi)
 \quad\text{for every }\varphi\in H_0^s(D).
\end{equation}
Set $u_j=w_j+g_j$.  Thus $u_j-g_j\in H_0^s(D)$ and $u_j$ is, by
construction, the energy solution of
\begin{equation}\label{eq:smooth-dirichlet-equivalence}
 \begin{cases}
 (-\Delta)^su_j+b_j\cdot\nabla u_j=f_j&\text{in }D,\\
 u_j=g_j&\text{in }\R^n\setminus D.
 \end{cases}
\end{equation}
Testing \eqref{eq:variational-approximation} with $w_j$ and using
\eqref{eq:drift-cancellation} gives a uniform $H^s(\R^n)$ bound for
$w_j$, hence for $u_j$.  For $k\ge\|g_j\|_{L^\infty}$, the truncation
$\eta=(u_j-k)^+$ belongs to $H_0^s(D)$ and satisfies
\[
 \mathcal E_s(u_j,\eta)\ge\mathcal E_s(\eta,\eta),
 \qquad \mathfrak d_{b_j}(u_j,\eta)=0.
\]
The second identity follows by writing $u_j=\eta+k$ on $\{\eta>0\}$
and using $\divergence b_j=0$.  Thus
$\mathcal E_s(\eta,\eta)\le\|f_j\|_\infty\|\eta\|_{L^1(D)}$.
The same argument applies to $(u_j+k)^-$; the fractional
Sobolev--Stampacchia iteration gives
\[
 \sup_j\|u_j\|_{L^\infty(\R^n)}
 \le C\left(\sup_j\|g_j\|_{L^\infty(\R^n)}
             +\sup_j\|f_j\|_{L^\infty(D)}\right).
\]
Finally, elliptic pseudodifferential regularity for
the smooth-coefficient equation gives $u_j\in C^\infty(D)$.  Hence
\eqref{eq:smooth-dirichlet-equivalence} holds pointwise, and $u_j$ is also
a viscosity solution.
Hence,
after taking a subsequence,
\begin{equation}\label{eq:two-compactness-modes}
 u_j\to v\quad\text{locally uniformly in }D,
 \qquad u_j\rightharpoonup v\quad\text{weakly in }H^s_{\mathrm{loc}}(D).
\end{equation}
The first convergence follows from the interior H\"older estimate and the
Arzel\`a--Ascoli theorem; the second follows from the energy bound.  We spell
out the nonlocal passage to the limit.  Given
$\sup_j\|u_j\|_{L^\infty(\R^n)}<\infty$, the local uniform convergence in
$D$, and the almost-everywhere convergence $u_j=g_j\to u$ outside $D$
also give, after a subsequence,
$u_j\to v$ in $L^1_{2s}(\R^n)$ by dominated convergence.  For
$\varphi\in C_c^\infty(D)$, choose $D_0\Subset D$ containing its support.
In the bilinear integral, weak $H^s(D_0)$ convergence treats the part over
$D_0\times D_0$.  \rev{On the complementary region, the support of
$\varphi$ keeps the kernel away from the diagonal.}  Convergence in
$L^1_{2s}(\R^n)$ supplies an integrable majorant.  The drift term converges
because $b_j\to b$ uniformly and $\varphi$ is smooth.  Hence the energy
identity passes to the limit.  Viscosity stability applied to the locally
uniform convergence shows independently that $v$ is a viscosity solution.
Thus the variationally constructed sequence has a limit carrying both
solution notions; no energy identity is inferred from a Perron solution.
\endgroup

\smallskip
\noindent\emph{Step 2: uniqueness in the energy class.}
Let $v_1,v_2$ be two energy weak solutions in $D$ with the same exterior
data and belonging to $H^s(\R^n)$, and put $z=v_1-v_2$.  Then
$z\in H_0^s(D)$.  Subtracting the two weak
identities and using smooth approximations of $z$ as test functions gives
\[
 \frac{c_{n,s}}2\iint_{\R^n\times\R^n}
 \frac{|z(x)-z(y)|^2}{|x-y|^{n+2s}}\,\differential x\differential y
 +\mathfrak d_b(z,z)=0.
\]
The second term is zero by \cref{lem:drift-form}.  The first term is
nonnegative, so it vanishes.  Hence $z$ is almost everywhere constant on
$\R^n$; because $z=0$ on $\R^n\setminus D$, this constant is zero.  Energy
solutions with fixed exterior data are therefore unique.

\smallskip
\noindent\emph{Step 3: weak implies viscosity.}
Suppose that $u$ is a weak solution.  In Step~1 choose $g_j$ converging to
the exterior values of $u$.  The weak limit $v$ has the same exterior data
and is a weak solution of the same equation.  Step~2 gives $v=u$ almost
everywhere.  Both are continuous in $D$, hence $v=u$ pointwise.  Since $v$
is a viscosity solution, so is $u$.

\smallskip
\noindent\emph{Step 4: viscosity implies weak.}
Suppose that $u$ is a viscosity solution and use the same exterior data in
Step~1.  The locally uniform limit $v$ is a viscosity solution of the same
Dirichlet problem.  The comparison principle for bounded viscosity
solutions gives $v=u$ in $D$.  Since Step~1 also shows that $v$ is an energy
weak solution, $u$ is a weak solution in $D$.  This completes both
implications.
\end{proof}

\begin{remark}[merely H\"older drifts]\label{rem:holder-equivalence}
If $b$ is only $C^\gamma$, the viscosity theorem remains valid as
stated, but the estimate used to extend $\mathfrak d_b$ to
$H_0^s\times H_0^s$ is not automatic.  In that regime the same weak
conclusion holds for weak solutions known a priori to arise from the smooth
approximation in Step~1, or, equivalently for the present purpose, for weak
solutions enjoying the corresponding comparison property.  We do not
identify an arbitrary distributional solution with a viscosity solution
without one of these additional hypotheses.
\end{remark}

\begin{remark}[integer intermediate exponents]
The fractional Schauder estimate is invoked at an auxiliary exponent
$\beta\in(\alpha,\gamma)$ chosen so that
$2s+\beta\notin\mathbb N$.  Such a choice is always possible,
and the resulting $C^{1,\beta}$ estimate implies the asserted
$C^{1,\alpha}$ estimate.  Thus no integer-exponent exclusion is needed in
\cref{thm:main}.  A Zygmund formulation would be needed only to claim an
exact integer intermediate $C^{2s+\beta}$ endpoint.
\end{remark}

\section{Measure data and potential estimates}\label{sec:measure}

\rev{We first establish the whole-space potential estimates and their limit
as $s\uparrow1$.  We then pass to the drifted Green kernel: probabilistic
harmonicity is linked to the viscosity formulation, the kernel is
differentiated by the interior estimate, and the resulting bounds are used
to construct the Green operator on measures.}

We consider
\begin{equation}\label{eq:measure-equation}
 (-\Delta)^s u+b\cdot\nabla u=\mu,
\end{equation}
where $\mu$ is a finite signed Radon measure.  The gradient estimate has a
positive potential order only when $s>1/2$; throughout
\cref{sec:measure}, unless explicitly stated otherwise, we therefore assume
\[
 \frac12<s<1,\qquad n>2s.
\]
The restriction $n>2s$ is automatic under the standing assumption $n\ge2$.

Recall from \eqref{eq:riesz-potential-main} the truncated Riesz
potential $\I_\beta^{|\mu|}(x,R)$.
For finite $R$, Fubini's theorem gives the equivalent kernel form
\begin{equation}\label{eq:riesz-kernel-equivalence}
 \I_\beta^{|\mu|}(x,R)
 \asymp
 \int_{B_R(x)}\frac{\differential|\mu|(y)}{|x-y|^{n-\beta}},
\end{equation}
with constants depending only on $n$ and $\beta$.  We write
$\I_\beta^{|\mu|}(x)$ for the untruncated potential defined above.
If $t=|x-y|<R$, then
\[
 \int_t^R\rho^{\beta-n-1}\,\differential\rho
 =\frac{t^{\beta-n}-R^{\beta-n}}{n-\beta}.
\]
Tonelli's theorem applied to this identity gives the kernel integral up to
the terminal contribution $R^{\beta-n}|\mu|(B_R(x))$; that contribution is
the second term in \eqref{eq:riesz-potential-main}.  This proves
\eqref{eq:riesz-kernel-equivalence} in both directions, including measures
concentrated near $\partial B_R(x)$.

The whole-space model fixes the orders of the two potentials.  We begin with
the proof of \cref{thm:whole-space-potential}.

\begin{proof}[Proof of \cref{thm:whole-space-potential}]
\begingroup
The kernel $\Phi_{n,s}$ is locally integrable, so $u\in L^1_{\mathrm{loc}}$.
Tonelli's theorem applies on each compact set because
$\mu$ is finite and compactly supported.  For
$\varphi\in C_c^\infty(\R^n)$, Fubini's theorem is legitimate: near
$x=y$ the fundamental kernel is locally integrable, while at infinity
$(-\Delta)^s\varphi(x)=O(|x|^{-n-2s})$.  The fundamental-solution identity
therefore gives
\begin{align*}
 \int_{\R^n}u(-\Delta)^s\varphi
 &=\int_{\R^n}\!\left[\int_{\R^n}
 \Phi_{n,s}(x-y)(-\Delta)^s\varphi(x)\,\differential x\right]
 \differential\mu(y)\\
 &=\int_{\R^n}\varphi\,\differential\mu.
\end{align*}
The layer-cake formula yields
\[
 \int_{\R^n}|x-y|^{-(n-\beta)}\,\differential|\mu|(y)
 =(n-\beta)\I_\beta^{|\mu|}(x)
\]
whenever either side is finite.  This proves the estimate for $u$.  Moreover,
$|\nabla\Phi_{n,s}(z)|\le C|z|^{-(n-(2s-1))}$.  To identify the gradient,
first truncate the convolution to $|x-y|>\varepsilon$.  The truncated
potential is differentiable, and its gradient is the convolution with
$\nabla\Phi_{n,s}$.  The latter kernels converge in
$L^1_{\mathrm{loc}}$ as $\varepsilon\downarrow0$, so their limit is the
distributional gradient of $u$.  At every point where
$\I_{2s-1}^{|\mu|}(x)<\infty$, the limiting integral is absolutely
convergent and satisfies the asserted pointwise bound.  Taking
$\mu=\delta_0$ gives $u(x)=\Phi_{n,s}(x)$ and proves sharpness of both
orders.
\endgroup
\end{proof}

\rev{For the limit $s\uparrow1$, kernel convergence is uniform away from
the pole; at points of $\operatorname{supp}\mu$ we use dominated
convergence.}

\begin{proof}[Proof of \cref{thm:s-to-one}]
For $n\ge3$, continuity of the Gamma function gives
\[
 \kappa_{n,s}\longrightarrow
 \frac{\Gamma((n-2)/2)}{4\pi^{n/2}}
 =\frac{1}{(n-2)|S^{n-1}|},
\]
while $|x|^{2s-n}\to|x|^{2-n}$ locally uniformly away from the origin.
This proves \eqref{eq:newton-limit}.  Away from
$\operatorname{supp}\mu$, the kernels and their first derivatives converge
uniformly on the relevant compact set, so the potential convergence is
locally uniform.  At a point in the support, take $s$ close enough to one
that $2s\ge2-\delta$ and $2s-1\ge1-\delta$.  On $|x-y|<1$ the kernels
are dominated by those in \eqref{eq:s-limit-dominating-potentials}, while
compact support controls $|x-y|\ge1$.  \rev{The kernel constants are
uniform for $s$ in a fixed interval below $1$, so the majorants are
independent of $s$.}  Dominated
convergence applies.

For $n=2$, put $\varepsilon=1-s$.  The expansions
\[
 \Gamma(\varepsilon)=\frac1\varepsilon-\gamma_E+O(\varepsilon),
 \qquad |x|^{-2\varepsilon}=1-2\varepsilon\log|x|+O(\varepsilon^2)
\]
and $2^{2s}\pi\Gamma(s)=4\pi+O(\varepsilon)$ yield
\[
 \Phi_{2,s}(x)=\frac{1}{4\pi\varepsilon}
 -\frac{1}{2\pi}\log|x|+c_0+O(\varepsilon(1+|\log|x||^2))
\]
locally uniformly away from the origin.  This proves
\eqref{eq:log-limit}; \rev{the error is uniform on compact subsets of
$\R^2\setminus\{0\}$.}
Differentiating the explicit kernel proves
\eqref{eq:gradient-log-limit}.  Finally,
\eqref{eq:weak-exponent-limit} is immediate.
\end{proof}

\begin{remark}[uniformity of constants near the local limit]
\label{rem:uniform-s}
The constants in the present paper are allowed to depend on $s$.  Thus the
results above prove consistency of the scaling, kernels, and mapping
exponents, but not a uniform estimate of the form
\[
 \sup_{s\in[s_0,1)}C_s<\infty
 \qquad(s_0>1/2).
\]
Such a result would require, in particular, Schauder estimates and Green
kernel derivative estimates uniform for $s\in[s_0,1)$, together with
uniform control of the drift perturbation.  The absorption mechanism itself
does not degenerate on this interval, since $2s-1\ge2s_0-1>0$ and bounded
drifts have a uniformly subcritical small-scale factor
$r^{2s-1}\le r^{2s_0-1}$ for $0<r\le1$.
\end{remark}

\rev{We now pass from the whole-space kernels to the bounded-domain
Dirichlet problem.  Green-function comparison gives the zeroth-order
kernel bound; a localized Dynkin argument then places the probabilistically
harmonic Green sections in the viscosity framework and permits the interior
estimate to yield the first-derivative bound.}  Let $D\subset\R^n$ be a
bounded $C^{1,1}$ domain.  For $\beta=2s-1$, the gradient Kato class is
\begin{equation}\label{eq:kato-measure-section}
 \mathcal K_n^{\beta}
 :=\left\{b:\ \lim_{r\downarrow0}\sup_{x\in\R^n}
 \int_{B_r(x)}\frac{|b(y)|}{|x-y|^{n-\beta}}\,\differential y=0\right\}.
\end{equation}
Every bounded drift belongs to this class.  When $b$ is given only on $D$,
membership in $\mathcal K_n^\beta$ refers to its zero extension to $\R^n$.
We use the following Green-kernel convention.  The operator acts in the
first variable,
\[
 \Lb^xG_b^D(x,y)=\delta_y\quad\text{in }D,
 \qquad G_b^D(\cdot,y)=0\quad\text{in }\R^n\setminus D,
\]
in the distributional sense.  Equivalently, for every
$h\in C_c^\infty(D)$, the function
\[
 \mathcal G_bh(x):=\int_DG_b^D(x,y)h(y)\,\differential y
\]
is the zero-exterior solution of $\Lb\mathcal G_bh=h$.  Since $\Lb$ is not
self-adjoint, the orientation of the variables is essential.  In
distributional form, the first-variable identity is
\begin{equation}\label{eq:green-adjoint-identity}
 \int_D G_b^D(x,y)(-\Delta)^s\varphi(x)\,\differential x
 +\int_D b(x)\cdot\nabla_xG_b^D(x,y)\,\varphi(x)\,\differential x
 =\varphi(y).
\end{equation}
\rev{Identity \eqref{eq:green-adjoint-identity} is used when both integrals
are defined.  Once \cref{lem:green-gradient} has been established, their
absolute convergence follows from \eqref{eq:green-upper} and
\eqref{eq:green-gradient}.}  This
formulation avoids pairing a rough solution with the distribution
$-\divergence(b\varphi)$.

For such gradient perturbations,
the Green function $G_b^D$ of the zero-exterior operator is comparable with
the Green function $G_0^D$ of the fractional Laplacian; see
\cite{BogdanJakubowski2012}.  \rev{More precisely, that reference uses the
generator convention $\Delta^{\alpha/2}+\widetilde b\cdot\nabla$, with
$\alpha=2s$ and $\Delta^{\alpha/2}=-(-\Delta)^s$; its result is therefore
applied here with $\widetilde b=-b$.}  In particular,
\begin{equation}\label{eq:green-upper}
 0\le G_b^D(x,y)\le C G_0^D(x,y)
 \le C|x-y|^{2s-n}.
\end{equation}

\begin{proposition}[Green-kernel framework]\label{prop:green-framework}
Under the assumptions above, the Green operator is positive and satisfies
\eqref{eq:green-upper}.  \rev{The kernel is jointly continuous off the
diagonal; for each $K\Subset D$, it extends continuously to
\[
 \{(x,y)\in K\times\overline D:x\ne y\},
\]
with value zero for $y\in\partial D$.}  For $h\in C_c^\infty(D)$, the
Green potential $\mathcal G_bh$ is the zero-exterior distributional solution of
$\Lb u=h$.  \rev{After \cref{lem:green-gradient}, the corresponding
section identity is represented by the absolutely convergent integrals in
\eqref{eq:green-adjoint-identity}.  Joint continuity is
\cite[Lemma~7]{BogdanJakubowski2012}; the comparison estimate is
\cite[Theorem~1]{BogdanJakubowski2012}; and the Green-operator identity is
based on the perturbation formula in
\cite[Section~3]{BogdanJakubowski2012}.  The asserted boundary extension
then follows from comparison with the drift-free Green kernel.  Moreover,
in the terminology of \cite[Definition~2]{BogdanJakubowski2012}, each
section $G_b^D(\cdot,y)$ is harmonic in $D\setminus\{y\}$ for the Markov
generator $-L_b$.}
\end{proposition}

\begingroup
The preceding probabilistic statement must be matched with the pasted-test
notion in \cref{def:viscosity}.  We first record the stopped identity that
allows this passage.  Let $\widetilde{\mathbb E}^{\,x}$ denote expectation
for the process with generator
$\mathcal A_b:=-L_b=-(-\Delta)^s-b\cdot\nabla$.

\begin{lemma}[localized Dynkin formula]\label{lem:localized-dynkin}
Let $U\Subset D$ be a ball, and let
$w\in L^1_{2s}(\R^n)$ be $C^2$ in a neighborhood of $\overline U$.
If $\mathcal A_bw$ is continuous on $\overline U$, then, for $x\in U$,
\begin{equation}\label{eq:localized-dynkin}
 \widetilde{\mathbb E}^{\,x}w(X_{\tau_U})-w(x)
 =\widetilde{\mathbb E}^{\,x}\int_0^{\tau_U}
   \mathcal A_bw(X_t)\,\differential t,
 \qquad
 \tau_U:=\inf\{t>0:X_t\notin U\}.
\end{equation}
Both sides of \eqref{eq:localized-dynkin} are finite.
\end{lemma}

\begin{proof}
Choose $U'$ with $\overline U\subset U'\Subset D$ such that $w$ is $C^2$
in a neighborhood of $\overline{U'}$.  A local extension, mollification,
and cutoff give $w_k\in C_c^\infty(\R^n)$ such that
\[
 w_k\longrightarrow w\quad\text{in }C^2(\overline{U'})
 \quad\text{and in }L^1_{2s}(\R^n).
\]
Consequently, the principal-value part is controlled by the local $C^2$
convergence and the nonlocal tail by the weighted $L^1$ convergence; thus
$\mathcal A_bw_k\to\mathcal A_bw$ uniformly on $\overline U$.  The
generator identity \cite[(36)]{BogdanJakubowski2012} and the strong Markov
property give, for each $k$,
\[
 \widetilde{\mathbb E}^{\,x}w_k(X_{\tau_U})-w_k(x)
 =\widetilde{\mathbb E}^{\,x}\int_0^{\tau_U}
   \mathcal A_bw_k(X_t)\,\differential t.
\]
The exit-distribution formula \cite[Lemmas~6 and 14]{BogdanJakubowski2012}
has no boundary mass.  On $U'\setminus U$, the $C^2$ convergence therefore
gives convergence of the exit terms, while the Poisson-kernel
representation \cite[(38)--(39)]{BogdanJakubowski2012} controls their far
part by the $L^1_{2s}$ norm.  Finally,
$\widetilde{\mathbb E}^{\,x}\tau_U<\infty$ by
\cite[Lemma~7]{BogdanJakubowski2012}.  Uniform convergence of the
generators on $U$ now permits passage to the limit in the stopped identity
and proves \eqref{eq:localized-dynkin}, including finiteness.
\end{proof}

\begin{proposition}[probabilistic harmonicity implies viscosity harmonicity]
\label{prop:green-viscosity}
Under the assumptions of \cref{prop:green-framework}, suppose in addition
that $b\in C^\gamma_{\mathrm{loc}}(D)$.  For every $y\in D$, the section
\[
 v(x):=G_b^D(x,y),\qquad v=0\quad\text{in }\R^n\setminus D,
\]
is a viscosity solution of
\begin{equation}\label{eq:green-viscosity-equation}
 (-\Delta)^sv+b\cdot\nabla v=0
 \qquad\text{in }D\setminus\{y\}.
\end{equation}
\end{proposition}

\begin{proof}
By \eqref{eq:green-upper}, the singularity at $y$ is locally integrable;
since $D$ is bounded and $v$ vanishes outside $D$, it follows that
$v\in L^1_{2s}(\R^n)$.  The function is continuous away from $y$ by
\cref{prop:green-framework}.  Moreover, its probabilistic harmonicity says
that for every ball $U\Subset D\setminus\{y\}$ and every $x\in U$,
\begin{equation}\label{eq:green-mean-value}
 v(x)=\widetilde{\mathbb E}^{\,x}v(X_{\tau_U}).
\end{equation}

Fix $x_0\in D\setminus\{y\}$, and let $\phi\in C^2(B_\rho(x_0))$ touch
$v$ from above there.  For an arbitrary $r\in(0,\rho)$, set
$w=\phi^{\,r,v}_{x_0}$ as in \eqref{eq:pasted-test}.  Thus $w\ge v$ on
$\R^n$ and $w(x_0)=v(x_0)$.  Choose
$0<\delta<r$ with $\overline{B_\delta(x_0)}\subset B_r(x_0)$.  From
\eqref{eq:green-mean-value} and \cref{lem:localized-dynkin},
\[
 0\le \widetilde{\mathbb E}^{\,x_0}
       w(X_{\tau_{B_\delta}})-w(x_0)
 =\widetilde{\mathbb E}^{\,x_0}\int_0^{\tau_{B_\delta}}
       \mathcal A_bw(X_t)\,\differential t.
\]
Because $\mathcal A_bw$ is continuous near $x_0$,
\[
 \left|\frac{\widetilde{\mathbb E}^{\,x_0}
 \int_0^{\tau_{B_\delta}}
 [\mathcal A_bw(X_t)-\mathcal A_bw(x_0)]\,\differential t}
 {\widetilde{\mathbb E}^{\,x_0}\tau_{B_\delta}}\right|
 \le \sup_{B_\delta(x_0)}
 |\mathcal A_bw-\mathcal A_bw(x_0)|\longrightarrow0.
\]
Dividing by $\widetilde{\mathbb E}^{\,x_0}\tau_{B_\delta}$ and letting
$\delta\downarrow0$ yields $\mathcal A_bw(x_0)\ge0$, or equivalently
$L_bw(x_0)\le0$.  This is precisely the subsolution inequality in
\cref{def:viscosity}.  A test function touching from below gives
$w\le v$ and reverses the inequalities, proving the supersolution
property.  Since $r$ was arbitrary, \eqref{eq:green-viscosity-equation}
holds in the pasted-test sense.
\end{proof}
\endgroup

\rev{With the probabilistic and viscosity notions now aligned, the interior
estimate applies to Green sections.  The following cutoff lemma records the
localized form and the tail term needed below.}

\begin{lemma}[localized gradient estimate with tail]
\label{lem:localized-gradient-tail}
Let $1/2<s<1$, $0<\alpha<\gamma<1$, and
$B_{4r}(x_0)\Subset D$.  Assume that
\[
 v\in C(B_{4r}(x_0))\cap L^1_{2s}(\R^n),
 \qquad b,g\in C^\gamma(B_{4r}(x_0)),
\]
and that $v$ is a viscosity solution of
\[
 (-\Delta)^sv+b\cdot\nabla v=g
 \qquad\text{in }B_{4r}(x_0).
\]
A classical solution is, in particular, covered.  Then
$v\in C^{1,\alpha}(B_{r/2}(x_0))$, and
\begin{align}\label{eq:localized-gradient-tail}
 |\nabla v(x_0)|\le C\Bigg[&r^{-1}\|v\|_{L^\infty(B_{4r}(x_0))}
 +r^{2s-1}\|g\|_{L^\infty(B_{4r}(x_0))}\notag\\
 &+r^{2s+\gamma-1}[g]_{C^\gamma(B_{4r}(x_0))}
 +r^{2s-1}\int_{\R^n\setminus B_{2r}(x_0)}
 \frac{|v(y)|}{|x_0-y|^{n+2s}}\,\differential y\Bigg].
\end{align}
The constant depends only on the structural parameters and the scaled
$C^\gamma$ norm of $b$.
\end{lemma}

\begin{proof}
Choose $\eta\in C_c^\infty(B_{3r}(x_0))$ such that
$0\le\eta\le1$, $\eta\equiv1$ on $B_{2r}(x_0)$, and
$|D^k\eta|\le C_kr^{-k}$.  Put $w=\eta v$.  Since $\nabla\eta=0$ in
$B_{2r}(x_0)$, a direct expansion of the fractional Laplacian gives, for
$x\in B_{2r}(x_0)$,
\[
 (-\Delta)^sw+b\cdot\nabla w
 =g+\mathcal C_\eta[v],
\qquad
 \mathcal C_\eta[v](x)
 =c_{n,s}\int_{\R^n}
 \frac{(1-\eta(y))v(y)}{|x-y|^{n+2s}}\,\differential y.
\]
The integral has no singularity because $1-\eta$ vanishes in
$B_{2r}(x_0)$.  For $x\in B_r(x_0)$ and $k=0,1$, differentiation under the
integral is justified by absolute convergence and yields
\[
 |D_x^k\mathcal C_\eta[v](x)|
 \le C\int_{\R^n\setminus B_{2r}(x_0)}
 \frac{|v(y)|}{|x_0-y|^{n+2s+k}}\,\differential y.
\]
Using the mean-value theorem for the kernel gives, for $x,x'\in B_r(x_0)$,
\[
 \frac{|\mathcal C_\eta[v](x)-\mathcal C_\eta[v](x')|}
 {|x-x'|^\gamma}
 \le Cr^{1-\gamma}
 \int_{\R^n\setminus B_{2r}(x_0)}
 \frac{|v(y)|}{|x_0-y|^{n+2s+1}}\,\differential y.
\]
Since $|x_0-y|\ge2r$ on the integration region, both estimates imply
\[
 \|\mathcal C_\eta[v]\|_{L^\infty(B_r)}
 +r^\gamma[\mathcal C_\eta[v]]_{C^\gamma(B_r)}
 \le C\int_{\R^n\setminus B_{2r}(x_0)}
 \frac{|v(y)|}{|x_0-y|^{n+2s}}\,\differential y.
\]
Moreover,
$\|w\|_{L^\infty(\R^n)}\le\|v\|_{L^\infty(B_{3r}(x_0))}$.
In $B_{2r}(x_0)$ the cutoff function is constant, so the preceding identity
holds in the viscosity sense.  The function $w$ is bounded on all of $\R^n$
because it has compact support and $v$ is continuous on $B_{4r}(x_0)$.
\rev{The commutator is smooth on compact subsets of $B_{2r}(x_0)$, so
\cref{thm:main} gives $v=w\in C^{1,\alpha}(B_{r/2}(x_0))$.  Applying the
same theorem on $B_r(x_0)$ with radius $r/2$, and using the bounds above,
gives \eqref{eq:localized-gradient-tail}.}
\end{proof}

\begin{lemma}[interior derivative of the Green kernel]
\label{lem:green-gradient}
Assume in addition that $b\in C^\gamma_{\mathrm{loc}}(D)$.  Let
$K\Subset D$.  Then, for $x\in K$, $y\in D$, and $x\ne y$,
\begin{equation}\label{eq:green-gradient}
 |\nabla_xG_b^D(x,y)|
 \le C_K |x-y|^{2s-n-1}.
\end{equation}
The constant depends on $n,s,\gamma,D,K$, the local $C^\gamma$ norm of $b$,
and its Kato modulus.
\rev{Moreover, $\nabla_xG_b^D$ extends continuously to
\[
 \{(x,y)\in K\times\overline D:x\ne y\},
\]
with value zero for $y\in\partial D$.}
\end{lemma}

\begin{proof}
\rev{Set
\[
 r=\frac1{16}\min\{|x-y|,\operatorname{dist}(K,\partial D)\}.
\]
Then $z\mapsto G_b^D(z,y)$ solves the homogeneous equation in
$B_{4r}(x)$ by \cref{prop:green-viscosity}; this ball is separated from
the pole and the boundary.}  Applying \cref{lem:localized-gradient-tail} with $g=0$ and
using that the Green function vanishes outside $D$ gives
\begin{align}\label{eq:green-gradient-local-step}
 |\nabla_xG_b^D(x,y)|
 &\le \frac{C}{r}\sup_{B_{4r}(x)}G_b^D(\cdot,y)\notag\\
 &\quad+C r^{2s-1}
 \int_{D\setminus B_{2r}(x)}
 \frac{G_b^D(z,y)}{|x-z|^{n+2s}}\,\differential z.
\end{align}
Suppose first that $|x-y|\le\operatorname{dist}(K,\partial D)$.  Then
\rev{$r=|x-y|/16$}.  By \eqref{eq:green-upper},
\[
 \sup_{B_{4r}(x)}G_b^D(\cdot,y)\le C|x-y|^{2s-n}.
\]
For the tail term, scaling $z=x+|x-y|w$ and using again
\eqref{eq:green-upper} give
\[
 \int_{D\setminus B_{2r}(x)}
 \frac{G_b^D(z,y)}{|x-z|^{n+2s}}\,\differential z
 \le C|x-y|^{-n}.
\]
Consequently, \eqref{eq:green-gradient-local-step} gives
$|\nabla_xG_b^D(x,y)|\le C|x-y|^{2s-n-1}$.

If $|x-y|>\operatorname{dist}(K,\partial D)=:d_K$, choose
\rev{$r=d_K/16$}.  \rev{Then $B_{4r}(x)\Subset D$ and
$|z-y|\ge |x-y|-|z-x|\ge3d_K/4$ for $z\in B_{4r}(x)$.}  Hence
$\sup_{B_{4r}(x)}G_b^D(\cdot,y)\le C_K$.  The tail integral in
\eqref{eq:green-gradient-local-step} is finite and bounded by $C_K$ by
splitting $D$ into $B_{d_K/4}(y)$ and its complement and using
\eqref{eq:green-upper}; near $y$ the factor $|x-z|^{-n-2s}$ is bounded,
whereas away from $y$ the Green kernel is bounded.  Thus
$|\nabla_xG_b^D(x,y)|\le C_K$.  Since in the present case
$|x-y|$ is bounded above by $\operatorname{diam}D$ and below by $d_K$,
this is equivalent, after changing $C_K$, to
\[
 |\nabla_xG_b^D(x,y)|\le C_K|x-y|^{2s-n-1}.
\]
This proves
\eqref{eq:green-gradient}.

\rev{It remains to prove the asserted joint continuity.  Fix an
off-diagonal point $(x,y)$ and a ball $B_{4\rho}(x)$ separated from $y$ and
$\partial D$.  If $y_j\to y$, \cref{prop:green-framework} gives
$G_b^D(\cdot,y_j)\to G_b^D(\cdot,y)$ uniformly on
$\overline{B_{4\rho}(x)}$.  The localized estimate, \eqref{eq:green-upper},
and the uniform separation from the poles give a uniform
$C^{1,\alpha}$ bound on $B_\rho(x)$.  Compactness in $C^1$ therefore yields
$\nabla_xG_b^D(\cdot,y_j)\to\nabla_xG_b^D(\cdot,y)$ locally uniformly.
The same argument for $y_j\to y\in\partial D$ uses the boundary extension
in \cref{prop:green-framework} and gives convergence of the gradients to
zero.}
\end{proof}

\rev{We now prove the Green-operator theorem stated in
\cref{thm:measure-potential}.}

\begin{proof}[Proof of \cref{thm:measure-potential}]
\begingroup
\emph{Step 1: extension of the Green operator.}
Define $u=\mathcal G_b\mu$ by \eqref{eq:green-potential-solution} wherever
the integral is absolutely convergent.  By \eqref{eq:green-upper}, Tonelli's
theorem, and the layer-cake identity,
\[
 |u(x)|\le C\int_D|x-y|^{2s-n}\,\differential|\mu|(y)
 \le C\I_{2s}^{|\mu|}(x,\operatorname{diam}D).
\]
This proves \eqref{eq:u-potential-bound} at every point where its right-hand
side is finite.  Moreover,
\[
 \int_D|u(x)|\,\differential x
 \le C\int_D\int_D|x-y|^{2s-n}\,\differential x\,
       \differential|\mu|(y)
 \le C|\mu|(D).
\]
Thus the potential is defined almost everywhere and belongs to $L^1(D)$.
Linearity is immediate, and positivity of $G_b^D$ shows that
$\mu\ge0$ implies $\mathcal G_b\mu\ge0$.

\emph{Step 2: identification and estimate of the weak gradient.}
Fix $K\Subset D$ and set
\[
 V_\mu(x):=\int_D\nabla_xG_b^D(x,y)\,\differential\mu(y)
\]
wherever this integral is absolutely convergent.  By
\cref{lem:green-gradient},
\[
 |V_\mu(x)|
 \le C_K\int_D|x-y|^{2s-n-1}\,\differential|\mu|(y)
 \le C_K\I_{2s-1}^{|\mu|}(x,2\operatorname{diam}D).
\]
This proves \eqref{eq:gradient-potential-bound}.  Since $2s-1>0$, Tonelli's
theorem also gives
\[
 \int_K|V_\mu(x)|\,\differential x\le C_K|\mu|(D).
\]
Together with Step~1, this is \eqref{eq:green-measure-operator-bound}.

To identify $V_\mu$ as the weak gradient, let
$\psi\in C_c^\infty(K;\R^n)$, take
$0<\varepsilon<\operatorname{dist}(K,\partial D)/2$, and remove the set
$|x-y|\le\varepsilon$.  Ordinary integration by parts away from the pole
gives
\begin{align*}
 &\int_D\int_{|x-y|>\varepsilon}
 G_b^D(x,y)\,\divergence\psi(x)
 \,\differential x\,\differential\mu(y)\\
 &\quad=-\int_D\int_{|x-y|>\varepsilon}
 \nabla_xG_b^D(x,y)\cdot\psi(x)
 \,\differential x\,\differential\mu(y)+E_\varepsilon.
\end{align*}
The boundary contribution is supported on $|x-y|=\varepsilon$ and obeys
\[
 |E_\varepsilon|
 \le C\varepsilon^{2s-1}|\mu|(D)\|\psi\|_{L^\infty}.
\]
The two volume integrals converge absolutely as
$\varepsilon\downarrow0$, by the kernel bounds of orders $2s$ and $2s-1$.
Consequently,
\[
 \int_Du\,\divergence\psi\,\differential x
 =-\int_DV_\mu\cdot\psi\,\differential x.
\]
Hence $u\in W^{1,1}_{\mathrm{loc}}(D)$ and
$V_\mu=\nabla u$ almost everywhere.

\emph{Step 3: the distributional equation.}
Let $\varphi\in C_c^\infty(D)$.  The Green identity
\eqref{eq:green-adjoint-identity} and Fubini's theorem yield
\begin{align*}
 &\int_Du(x)(-\Delta)^s\varphi(x)\,\differential x
 +\int_Db(x)\cdot\nabla u(x)\,\varphi(x)\,\differential x\\
 &\quad=\int_D\left[\int_DG_b^D(x,y)(-\Delta)^s\varphi(x)
   \,\differential x\right.\\
 &\hspace{5.7em}\left.
   +\int_Db(x)\cdot\nabla_xG_b^D(x,y)\,\varphi(x)
   \,\differential x\right]\differential\mu(y)\\
 &\quad=\int_D\varphi(y)\,\differential\mu(y).
\end{align*}
Both interchanges of integration are absolutely justified:
$(-\Delta)^s\varphi$ is bounded on $D$, \eqref{eq:green-upper} controls
the first kernel, and $b$ is bounded on $\operatorname{supp}\varphi$ while
\cref{lem:green-gradient} controls the second.  This proves
\eqref{eq:measure-distributional}.

\emph{Step 4: weak-* stability and the SOLA characterization.}
Suppose that $\mu_j\stackrel{*}{\rightharpoonup}\mu$ in
$\M(\overline D)$ and that
$M:=\sup_j|\mu_j|(D)<\infty$.  Fix $K\Subset D$.  Choose
$\eta\in C^\infty([0,\infty))$ with $0\le\eta\le1$,
$\eta=0$ on $[0,1]$, and $\eta=1$ on $[2,\infty)$.  For
$\varepsilon>0$ define the smoothly truncated kernels
\[
 G_\varepsilon(x,y)
 :=\eta\!\left(\frac{|x-y|}{\varepsilon}\right)G_b^D(x,y),
 \qquad
 H_\varepsilon(x,y)
 :=\eta\!\left(\frac{|x-y|}{\varepsilon}\right)
   \nabla_xG_b^D(x,y).
\]
By \cref{prop:green-framework,lem:green-gradient}, these kernels extend
continuously and boundedly to $K\times\overline D$.  The
Stone--Weierstrass theorem gives uniform approximation there by finite sums
of functions of $x$ times functions of $y$.  Weak-* convergence then gives
\begin{align*}
 \sup_{x\in K}\left|
  \int_DG_\varepsilon(x,y)\,\differential(\mu_j-\mu)(y)
 \right|&\longrightarrow0,\\
 \sup_{x\in K}\left|
  \int_DH_\varepsilon(x,y)\,\differential(\mu_j-\mu)(y)
 \right|&\longrightarrow0.
\end{align*}
The discarded near-diagonal parts satisfy, uniformly in $j$,
\begin{align*}
 \int_K\int_{|x-y|<2\varepsilon}G_b^D(x,y)
  \,\differential|\mu_j|(y)\,\differential x
 &\le C M\varepsilon^{2s},\\
 \int_K\int_{|x-y|<2\varepsilon}|\nabla_xG_b^D(x,y)|
  \,\differential|\mu_j|(y)\,\differential x
 &\le C_K M\varepsilon^{2s-1},
\end{align*}
and the same bounds hold with $\mu$ in place of $\mu_j$.  First letting
$j\to\infty$ and then $\varepsilon\downarrow0$ proves
\eqref{eq:green-measure-stability}.

Every finite measure on $D$ admits an admissible smooth approximation:
apply mollification to its positive and negative parts on an exhaustion of
$D$ and take a diagonal sequence.  Stability then proves existence of a
Green-potential SOLA and convergence in
$W^{1,1}_{\mathrm{loc}}(D)$ for every admissible approximation.  Conversely,
the limit in the definition of a Green-potential SOLA must equal
$\mathcal G_b\mu$ by the same stability statement.  This proves uniqueness.

Finally, optimality already occurs in the drift-free subclass.  Take
$b=0$, $D=B_1$, and $\mu=\delta_0$.  Near the pole,
\[
 G_0^{B_1}(x,0)=\kappa_{n,s}|x|^{2s-n}-H(x),
\]
where $H$ is smooth near the origin.  Thus
$G_0^{B_1}(x,0)\asymp|x|^{2s-n}$ and
$|\nabla_xG_0^{B_1}(x,0)|\asymp|x|^{2s-n-1}$ as $x\to0$.
Neither potential order can therefore be improved over the stated class.
\endgroup
\end{proof}

The standard mapping properties of Riesz potentials immediately give the
following scale of consequences.

\begin{corollary}[weak Lebesgue and Lorentz estimates]
\label{cor:measure-mapping}
Under the assumptions of \cref{thm:measure-potential}, for every
$K\Subset D$,
\begin{equation}\label{eq:weak-estimates}
 u\in L^{\frac n{n-2s},\infty}(D),\qquad
 \nabla u\in L^{\frac n{n-2s+1},\infty}(K),
\end{equation}
with norms bounded by $C_K|\mu|(D)$.  If $\mu=f\,\differential x$ and
$1<q<n/(2s-1)$, then
\begin{equation}\label{eq:strong-gradient-mapping}
 \|\nabla u\|_{L^{q_*}(K)}\le C_K\|f\|_{L^q(D)},
 \qquad q_*:=\frac{nq}{n-(2s-1)q}.
\end{equation}
At the endpoint,
\begin{equation}\label{eq:lorentz-endpoint}
 f\in L^{\frac n{2s-1},1}(D)
 \quad\Longrightarrow\quad
 \nabla u\in L^\infty_{\mathrm{loc}}(D).
\end{equation}
\end{corollary}

\begin{proof}
These are the classical weak-type, strong-type, and Lorentz-endpoint
Riesz-potential mappings
\[
 \M\longrightarrow L^{n/(n-\beta),\infty},\qquad
 L^q\longrightarrow L^{nq/(n-\beta q)},\qquad
 L^{n/\beta,1}\longrightarrow L^\infty.
\]
Apply them with $\beta=2s$ and
$\beta=2s-1$ to \eqref{eq:u-potential-bound} and
\eqref{eq:gradient-potential-bound}.
\end{proof}

\begin{corollary}[Morrey and continuity criteria]
\label{cor:measure-morrey}
Assume the hypotheses of \cref{thm:measure-potential}.  Suppose that
$\theta\ge0$ and that, for every $K\Subset D$, there exist constants
$M_K,r_K>0$ such that
\begin{equation}\label{eq:morrey-measure}
 |\mu|(B_r(x))\le M_K r^{n-\theta}
 \qquad\text{for }x\in K\text{ and }0<r<r_K.
\end{equation}
\rev{Then} $u$ is locally bounded when $\theta<2s$, and $\nabla u$ is locally
bounded when $\theta<2s-1$.  If
\begin{equation}\label{eq:vanishing-gradient-potential}
 \lim_{r\downarrow0}\sup_{x\in K}
 \I_{2s-1}^{|\mu|}(x,r)=0
 \qquad\text{for every }K\Subset D,
\end{equation}
then $\nabla u$ has a continuous representative in $D$.
\end{corollary}

\begin{proof}
Inserting \eqref{eq:morrey-measure} into
\eqref{eq:riesz-potential-main} gives
$\I_\beta^{|\mu|}(x,R)\le CM_K R^{\beta-\theta}$ for sufficiently small
$R$ whenever $\theta<\beta$.  On each compact subset, the complementary
large-scale part is uniformly bounded because $\mu$ is finite.  This proves
the boundedness statements.  For continuity,
fix $x_0\in K$ and choose $r>0$ so that
$B_{4r}(x_0)\Subset D$.  For $x\in B_{r/2}(x_0)$ use the fixed splitting
\[
 D=B_{2r}(x_0)\cup\bigl(D\setminus B_{2r}(x_0)\bigr).
\]
The contribution of $B_{2r}(x_0)$ to both $\nabla u(x)$ and
$\nabla u(x_0)$ is bounded by
\[
 C\sup_{z\in B_r(x_0)}\I_{2s-1}^{|\mu|}(z,4r),
\]
which tends to zero uniformly as $r\downarrow0$ by
\eqref{eq:vanishing-gradient-potential}.  \rev{Since
$B_{2r}(x_0)\subset B_{4r}(x)$ for $x\in B_{r/2}(x_0)$, the same
truncated potential controls both near-field integrals.}  On the fixed far region,
$\nabla_xG_b^D(x,y)$ is continuous in $x$ by
\cref{lem:green-gradient}.  Moreover,
\cref{lem:green-gradient} supplies a bounded integrable majorant there.
Dominated convergence therefore gives continuity of the far contribution
at $x_0$.  Since $x_0$ is arbitrary, $\nabla u$ has a continuous
representative in $D$.
\end{proof}

\begin{remark}[the critical endpoint]
When $s=1/2$, the formal gradient order is $2s-1=0$.  Already for the
whole-space equation,
\[
 \nabla(-\Delta)^{-1/2}\mu
\]
is a vector Riesz transform of $\mu$, hence a zero-order Calder\'on--Zygmund
operator.  Thus no estimate by a positive-order potential is available in
general.  For finite measures the natural endpoint conclusion is weak
$L^1$, subject in the drifted case to the corresponding critical singular
integral theory.  This is analytically distinct from the positive-potential
estimate \eqref{eq:gradient-potential-bound}, and it is not claimed here for
a general critical drift.
\end{remark}

\begin{remark}[relation with nonlinear potential theory]
For nonlocal equations with measure data, nonlinear Wolff-potential
estimates and their Lorentz and continuity consequences were developed by
Kuusi, Mingione and Sire \cite{KuusiMingioneSire2015}.  In the linear case
$p=2$, the Wolff potential reduces, up to the standard normalization, to the
Riesz potentials appearing above.  The point of
\cref{thm:measure-potential} is that the drifted Green kernel preserves the
same sharp orders in the subcritical regime.
\end{remark}

\section{Boundary regularity and the effect of the drift}
\label{sec:boundary}

\rev{We recall the drift-free boundary estimate, apply it when
$\operatorname{supp}b\Subset\Omega$, and discuss the obstruction when the
drift reaches the boundary.}

\begin{proposition}[drift-free boundary regularity]\label{prop:boundary-base}
Let $0<s<1$, let $\Omega\subset\R^n$ be a bounded $C^{1,1}$ domain, and let
$f\in L^\infty(\Omega)$.  Suppose that $u$ is a bounded viscosity solution of
\begin{equation}\label{eq:dirichlet-pure}
 \begin{cases}
 (-\Delta)^su=f&\text{in }\Omega,\\
 u=0&\text{in }\R^n\setminus\Omega.
 \end{cases}
\end{equation}
Let $d(x)=\operatorname{dist}(x,\R^n\setminus\Omega)$ in $\Omega$ and extend
$d$ by zero in $\R^n\setminus\Omega$.  Then $u\in C^s(\R^n)$ and, for every
$\varepsilon\in(0,s)$,
\begin{equation}\label{eq:boundary-estimate}
 \|u\|_{C^s(\R^n)}
 +\left\|\frac{u}{d^s}\right\|_{C^{s-\varepsilon}(\overline\Omega)}
 \le C\|f\|_{L^\infty(\Omega)},
\end{equation}
where $C$ depends only on $n,s,\varepsilon,\Omega$ and the normalization of
the fractional Laplacian.
\end{proposition}

\begin{proof}
The fractional Laplacian is a symmetric stable operator whose spectral
measure is smooth and satisfies the ellipticity assumptions in
\cite{RosOtonSerra2016}.  Therefore \eqref{eq:boundary-estimate} is
the fractional-Laplacian specialization of
\cite[Theorem~1.2]{RosOtonSerra2016}.  The global $C^s$ estimate and the
$C^{s-\varepsilon}$ estimate for the quotient are both part of that theorem.
\end{proof}

The interior estimate proved above allows a genuine drift term, provided it
does not reach the boundary.

\begin{proof}[Proof of \cref{thm:boundary-drift}]
Choose open sets
\[
 K\Subset U_1\Subset U_2\Subset\Omega.
\]
Cover $\overline{U_1}$ by finitely many balls to which
\cref{thm:main} applies.  \rev{The radii may be chosen uniformly from below
in terms of $\operatorname{dist}(U_1,\partial U_2)$, so the finite-cover
constant depends only on the quantities stated in the theorem.}  For any fixed admissible interior exponent
$\alpha\in(0,\gamma)$ this gives
\begin{equation}\label{eq:gradient-on-support}
 \|\nabla u\|_{L^\infty(K)}
 \le C\bigl(
 \|u\|_{L^\infty(\R^n)}+
 \|f\|_{C^\gamma(U_2)}\bigr).
\end{equation}
The interior estimate, applied locally at every point of $\Omega$, makes
$\nabla u$ pointwise well defined throughout $\Omega$.  Define simply
\[
 g=f-b\cdot\nabla u\qquad\text{in }\Omega.
\]
Since $b=0$ in $\Omega\setminus K$, only the values of $\nabla u$ on $K$
enter this expression.  Hence
\eqref{eq:gradient-on-support} implies
\begin{align}\label{eq:g-boundary}
 \|g\|_{L^\infty(\Omega)}
 &\le \|f\|_{L^\infty(\Omega)}
 +\|b\|_{L^\infty(K)}
   \|\nabla u\|_{L^\infty(K)}\notag\\
 &\le C\bigl(
 \|u\|_{L^\infty(\R^n)}+
 \|f\|_{C^\gamma(\Omega)}\bigr).
\end{align}
\rev{If a $C^2$ test function $\phi$ touches $u$ at $x$, differentiability
gives $\nabla\phi(x)=\nabla u(x)$.  Thus the viscosity inequalities for the
original equation are those for $(-\Delta)^su=g$.}  Hence $u$ solves
\[
 \begin{cases}
 (-\Delta)^su=g&\text{in }\Omega,\\
 u=0&\text{in }\R^n\setminus\Omega,
 \end{cases}
\]
with $g\in L^\infty(\Omega)$.  Applying \cref{prop:boundary-base} and then
\eqref{eq:g-boundary} proves \eqref{eq:boundary-drift-estimate}--
\eqref{eq:boundary-drift-norm}.
\end{proof}

Without the support separation in \cref{thm:boundary-drift}, one cannot
apply \cref{prop:boundary-base} by simply writing
\[
 (-\Delta)^su=f-b\cdot\nabla u.
\]
Indeed, the expected boundary size $u\sim d^s$ gives
$|\nabla u|\sim d^{s-1}$, so the new right-hand side need not be bounded.
Thus the hypotheses of \cref{prop:boundary-base} are not preserved by this
rewriting.

There are two natural ways to weaken the support condition in
\cref{thm:boundary-drift}.  First, in the subcritical range $s>1/2$, one may
place the drift in the gradient Kato class
\begin{equation}\label{eq:kato}
 \lim_{r\downarrow0}\sup_{x\in\R^n}
 \int_{B_r(x)}
 \frac{|b(y)|}{|x-y|^{n+1-2s}}\,\differential y=0.
\end{equation}
Every bounded drift satisfies \eqref{eq:kato}, since the integral is bounded
by $C\|b\|_\infty r^{2s-1}$.  For such gradient perturbations, sharp Green
function estimates and a boundary Harnack principle with the $d^s$ decay
rate are known; see Chen, Kim and Song \cite{ChenKimSong2012}.  These results
strongly support a drifted boundary theory when $s>1/2$.  They do not by
themselves give the full inhomogeneous
$C^{s-\varepsilon}$ norm in \eqref{eq:boundary-drift-norm}; obtaining that
estimate requires a quantitative boundary Schauder argument for the Green
potential.

Second, one may impose a geometric decay condition on the normal drift.  A
natural scale is
\begin{equation}\label{eq:normal-decay}
 |b(x)\cdot\nabla d(x)|\le C d(x)^\kappa,
 \qquad \kappa\ge1-s.
\end{equation}
Indeed, if $u=d^sq$ near the boundary, the leading normal part of the drift
is
\[
 s d^{s-1}q\,b\cdot\nabla d,
\]
which is bounded under \eqref{eq:normal-decay}.  The stronger assumptions
$b\in C^1(\overline\Omega)$ and $b=0$ on $\partial\Omega$ imply
$|b(x)|\le C d(x)$ and hence satisfy this scaling requirement.  Condition
\eqref{eq:normal-decay} identifies a plausible weighted-Schauder extension;
we record it as a direction for further work, not as a theorem proved here.

The issue is especially substantial at $s=1/2$, where the drift is of the
same order as the diffusion.  In critical problems the boundary or free
boundary homogeneity can depend on the normal component of the drift; this
phenomenon is explicit in the critical-drift analysis of Fern\'andez-Real
and Ros-Oton \cite{FernandezRealRosOton2018}.  Consequently a universal
$d^{1/2}$ profile, and hence a universal estimate for $u/d^{1/2}$, should not
be asserted for \eqref{eq:equation} with an arbitrary critical drift.

When $s>1/2$, scaling makes the drift lower order and suggests that $d^s$
remains the leading boundary profile.  A proof, however, requires weighted
boundary Schauder estimates or a boundary blow-up argument controlling the
singular term $b\cdot\nabla u$; it does not follow from the interior
absorption argument in \cref{sec:proof}.  We therefore state
the drifted estimate under the support-separation hypothesis of
\cref{thm:boundary-drift}, and leave a sharp theorem for drifts reaching the
boundary outside the scope of this paper.

\subsection*{Acknowledgments}
This work was supported by the National Natural Science Foundation of China
(No.~12471128).

\subsection*{Conflict of interest}
The authors declare that there is no conflict of interest.

\subsection*{Data availability}
Data sharing is not applicable to this article, as no datasets were generated
or analyzed.

\end{document}